\pdfoutput=1  
\documentclass[letterpaper,10pt]{article}
\usepackage{fullpage}

\usepackage{graphicx}%
\usepackage{multirow}%
\usepackage{amsmath,amssymb,amsfonts}%
\usepackage{amsthm}%
\usepackage{mathrsfs}%
\usepackage[title]{appendix}%
\usepackage{xcolor}%
\usepackage{textcomp}%
\usepackage{manyfoot}%
\usepackage{booktabs}%
\usepackage{algpseudocode}%
\usepackage{listings}%
\usepackage{caption}
\usepackage[authoryear,round]{natbib}
\setcitestyle{authoryear,round,citesep={;},aysep={,},yysep={;}}
\usepackage{hyperref}
\usepackage[capitalize,nameinlink,noabbrev]{cleveref}
\usepackage{mathtools}

\usepackage{comment}
\usepackage{siunitx}
\usepackage{relsize}
\usepackage{ifthen}
\usepackage[colorinlistoftodos]{todonotes}

\usepackage[caption=false]{subfig}

\usepackage[vlined,ruled,linesnumbered]{algorithm2e}
\usepackage{graphics} 
\usepackage{rotating}
\usepackage{color}
\usepackage{enumerate}
\usepackage[T1]{fontenc}
\usepackage{psfrag}
\usepackage{epsfig} 
\usepackage{booktabs}
\usepackage{graphicx,url}
\usepackage{multirow}
\usepackage{array}
\usepackage{latexsym}
\usepackage{amsfonts}
\usepackage{amsmath}
\usepackage{amssymb}
\usepackage{mathtools}
\usepackage{xstring}
\usepackage{multirow}
\usepackage{xcolor}
\usepackage{prettyref}
\usepackage{bigdelim}
\usepackage{listings}

\usepackage{enumitem}
\usepackage{xspace}
\usepackage{bm}
\usepackage{tikz}
\usetikzlibrary{matrix,calc}

\usepackage{amsmath,array,arydshln,xparse}
\usepackage{adjustbox}

\usepackage{mdwlist}

\makecompactlist{itemize}{stditemize}

\newrefformat{prob}{Problem\,\ref{#1}}
\newrefformat{def}{Definition\,\ref{#1}}
\newrefformat{sec}{Section\,\ref{#1}}
\newrefformat{sub}{Section\,\ref{#1}}
\newrefformat{prop}{Proposition\,\ref{#1}}
\newrefformat{app}{Appendix\,\ref{#1}}
\newrefformat{alg}{Algorithm\,\ref{#1}}
\newrefformat{cor}{Corollary\,\ref{#1}}
\newrefformat{thm}{Theorem\,\ref{#1}}
\newrefformat{lem}{Lemma\,\ref{#1}}
\newrefformat{fig}{Fig.\,\ref{#1}}
\newrefformat{tab}{Table\,\ref{#1}}

\newcommand{\bdmath}{\begin{dmath}}
\newcommand{\edmath}{\end{dmath}}
\newcommand{\beq}{\begin{equation}}
\newcommand{\eeq}{\end{equation}}
\newcommand{\bdm}{\begin{displaymath}}
\newcommand{\edm}{\end{displaymath}}
\newcommand{\bea}{\begin{eqnarray}}
\newcommand{\eea}{\end{eqnarray}}
\newcommand{\beal}{\beq \begin{array}{ll}}
\newcommand{\eeal}{\end{array} \eeq}
\newcommand{\beas}{\begin{eqnarray*}}
\newcommand{\eeas}{\end{eqnarray*}}
\newcommand{\ba}{\begin{array}}
\newcommand{\ea}{\end{array}}
\newcommand{\bit}{\begin{itemize}}
\newcommand{\eit}{\end{itemize}}
\newcommand{\ben}{\begin{enumerate}}
\newcommand{\een}{\end{enumerate}}

\newcommand{\hide}[1]{}

\newcommand{\hiddenText}{{\color{gray} hidden text.}}
\newcommand{\hideWithText}[1]{\hiddenText}

\newcommand{\diag}[1]{\mathrm{diag}\left(#1\right)}

\newcommand{\blue}[1]{{\color{blue}#1}}

\newcommand{\linkToPdf}[1]{\href{#1}{\blue{(pdf)}}}
\newcommand{\linkToPpt}[1]{\href{#1}{\blue{(ppt)}}}
\newcommand{\linkToCode}[1]{\href{#1}{\blue{(code)}}}
\newcommand{\linkToWeb}[1]{\href{#1}{\blue{(web)}}}
\newcommand{\linkToVideo}[1]{\href{#1}{\blue{(video)}}}
\newcommand{\linkToMedia}[1]{\href{#1}{\blue{(media)}}}
\newcommand{\award}[1]{\xspace} 

\hypersetup{colorlinks, hypertexnames=false, pageanchor=true,
            linkcolor=blue, citecolor={green!50!black}, urlcolor=cyan,
            pdftitle={Global Extendability of Local Atoms in Sparse Moment Relaxations}
            }
\makeatletter
\AddToHook{cmd/appendix/before}{\def\cref@section@alias{appendix}}
\makeatother
\mathtoolsset{centercolon}
\crefname{assumption}{Assumption}{Assumptions}

\theoremstyle{plain}%
\newtheorem{theorem}{Theorem}%
\newtheorem{lemma}[theorem]{Lemma}%
\newtheorem{proposition}[theorem]{Proposition}%
\newtheorem{corollary}[theorem]{Corollary}%

\theoremstyle{remark}%

\theoremstyle{definition}%
\newtheorem{definition}{Definition}%
\newtheorem{example}{Example}%

\graphicspath{{figs/}}

\title{Global Extendability of Local Atoms in Sparse Moment Relaxations}
\author{%
  Haoran Sun\thanks{School of Engineering and Applied Sciences, Harvard University. Email: haoran\_sun@seas.harvard.edu}
}
\date{}  

\begin{document}

\maketitle

\begin{abstract}
Moment relaxations provide lower bounds for polynomial optimization problems, while their moment matrices can also carry information about solutions.
We study whether atomic candidates extracted from clique blocks reveal part of a global minimizer in moment relaxations with correlative sparsity.
We give sufficient conditions for global extendability. If the relaxation is value-exact and the moment matrix of a selected clique has maximum rank over the optimal face and satisfies simplicial regularity, then at least one of its atomic candidates is globally extendable.
When this maximum rank is one, we derive a direct certificate for global extendability.
Further examples demonstrate the necessity of value exactness for global extendability.
Numerical experiments on unicycle trajectory optimization successfully yield near-optimal solutions through progressive extraction and fixing, demonstrating the practical effectiveness of the proposed approach.
\end{abstract}

\section{Introduction}
\label{sec:introduction}

Global polynomial optimization with Lasserre's hierarchy of moment and sums-of-squares (SOS) relaxations~\citep{lasserre2001} plays an increasingly important role in control and robotics, from nonlinear optimal control~\citep{lasserre2007control} to trajectory optimization and contact-rich planning~\citep{kang2024strom,kang2025spot}. 
The hierarchy generates a sequence of convex semidefinite programs (SDPs) whose optimal values are nondecreasing \emph{lower bounds} on the global minimum. 
In many practical questions, the structure of the problem often have correlative sparsity, which can be exploited to reduce the size of the SDPs~\citep{waki2005,lasserre2006sparse,magron2021tssos}. It replaces the large single moment matrix by many small ones indexed by variable cliques.
In these applications, however, a bound is not enough: one needs a \emph{solution}, that is, a point that attains the optimal value. 
For dense moment relaxations, flat truncation is a classical certificate for minimizer extraction. More recently, simplicial regularity of the pseudo-moment cone has provided another route to atomic decomposition~\citep{kang2026simplicial}. In a sparse relaxation, however, only selected clique blocks may satisfy the extraction condition. We ask what these blocks can reveal about global minimizers without requiring the remaining blocks to satisfy the same condition.

\paragraph{Problem statement.}
We consider polynomial optimization problems (POPs) with correlative sparsity: the objective is a sum of terms, and each term and each constraint involves only a few of the variables. Such a POP takes the form
\begin{equation}
    f^\star=\min_{x\in\mathbb R^n}\ \Bigl\{\,\sum_{\ell=1}^{p}f_\ell(x_{\Delta_\ell})
    \ :\ x_{\Delta_\ell}\in K_\ell,\ \ell\in[p]\,\Bigr\},
    \label{eq:intro-pop}
\end{equation}
where the \emph{cliques} $\Delta_1,\ldots,\Delta_p\subseteq[n]$ are small, overlapping sets of variables, each $f_\ell$ is a polynomial in $x_{\Delta_\ell}$, and each $K_\ell$ is a basic semialgebraic set described by polynomial constraints in $x_{\Delta_\ell}$. Let $X^\star$ denote the set of global minimizers. Trajectory optimization is a prime example, because the discretized dynamics couple only consecutive time steps (Example~\ref{ex:chain}). The sparse moment relaxation of~\eqref{eq:intro-pop} at order $\kappa$ optimizes over a sequence of pseudo-moments $y$ with one moment matrix $M_\kappa^{\Delta_\ell}(y)$ per clique, and overlapping cliques share the moments of their common variables (§\ref{subsec:sparse-moments}). Its optimal value $\rho_\kappa$ is a lower bound on $f^\star$, and we call the relaxation \emph{value exact} if $\rho_\kappa=f^\star$.

Suppose that an optimal $y$ has a clique block $M_\kappa^\Delta(y)$ satisfying simplicial regularity (Definition~\ref{def:simplicial-face}): its minimal face in the pseudo-moment matrix cone is simplicial and generated by finite point evaluations.
The clique moments through degree $2\kappa$ then have a unique atomic decomposition. We call its points \emph{atomic candidates}; if every candidate satisfies the assigned clique constraints, the decomposition is a locally feasible atomic representation (Definition~\ref{def:local-representation}).
Neither the decomposition nor local feasibility alone ensures that a candidate $a$ is part of a global minimizer.
We call $a$ \emph{globally extendable} if $a\in\pi_\Delta(X^\star)$, that is, if some global minimizer agrees with $a$ on $\Delta$. 
Fixing $x_\Delta=a$ then retains at least one global minimizer. 

\paragraph{Extraction methods.}
Classical dense extraction uses flat truncation~\citep{nie2013flat}, followed by multiplication matrices and eigenvalue decompositions~\citep{henrion2005,klep2018}. For sparse relaxations~\citep{waki2005,lasserre2006sparse}, a classical sufficient condition combines flatness of every clique block with rank one of every overlap block~\citep[Theorem~3.7]{lasserre2006sparse}. Simplicial regularity instead yields atomic candidates from the facial geometry of a selected block, without a flatness assumption~\citep{kang2026simplicial}. When all clique blocks admit locally feasible atomic representations and their overlaps have rank one, the candidates can likewise be assembled into global minimizers (§\ref{subsec:flat-truncation}). Such a criterion still requires extraction on every clique; our focus is what can be certified when only selected cliques satisfy simplicial regularity.

\begin{quote}\itshape
If only selected clique blocks of an optimal sparse moment solution satisfy simplicial regularity, under what additional conditions do their atomic candidates extend to global minimizers?
\end{quote}

\paragraph{Contributions.}
In this paper, we show that simplicial regularity alone is not enough and identify what is missing. The key observation is that a clique interacts with the rest of the problem only through the variables it shares with other cliques, which we call its \emph{boundary}. We summarize our contributions as follows.
\begin{enumerate}[label=(\Roman*),leftmargin=*]
    \item \textbf{A selected simplicially regular clique can reveal part of a global minimizer.} A value-exact chain example (§\ref{sec:partial-flatness}) has a rank-one simplicially regular clique whose unique candidate agrees with every global minimizer, although the endpoint blocks are not simplicially regular. Conversely, Example~\ref{ex:gap-rank-one} gives a two-clique problem with an arbitrarily small positive relaxation gap whose selected rank-one clique has a locally feasible but non-extendable candidate. Example~\ref{ex:exact-rank-two} at $\theta=0$ demonstrates the same failure under value exactness. Thus value exactness and simplicial regularity at a single optimal solution do not guarantee globally extendable candidates.
    \item \textbf{Existence guarantees over the optimal face.} For a value-exact relaxation, we show that if a clique block has maximum rank over the optimal face of the relaxation and satisfies simplicial regularity, in the sense of Definition~\ref{def:simplicial-face}, then at least one of its atomic candidates is the projection of a global minimizer. This candidate is necessarily locally feasible. The rank-one case over the optimal face pins down the unique globally extendable assignment.
    \item \textbf{Atomic replacement and a boundary rank-one certificate.} Under the stated sparsity assumptions, every positive-weight atom in a locally feasible representation of an optimal clique block through the full relaxation degree minimizes the clique objective conditional on its boundary value. Simplicial regularity supplies such a representation when all candidates satisfy the assigned constraints. Consequently, an atom is globally extendable as soon as its boundary value occurs in a global minimizer. When the relaxation is value exact and the boundary moment block has maximum rank one over the whole optimal face, every extracted atom of the clique is globally extendable.
    Four examples (§\ref{sec:extendability-examples}) illustrate the results and show why maximum rank, value exactness, and the boundary rank-one condition are needed.An auxiliary SDP with the same semidefinite blocks as the relaxation tests the rank-one hypothesis over the optimal face (§\ref{sec:unicycle}).
\end{enumerate}

\paragraph{Paper organization.}
We review dense and sparse moment relaxations and atomic solution extraction in §\ref{sec:preliminaries}, retaining flat truncation as classical background. We present the motivating example in §\ref{sec:partial-flatness}. In §\ref{sec:global-extendability}, we define simplicial regularity, global extendability, and clique boundaries, and establish sufficient conditions for extending atomic candidates to global minimizers.
In §\ref{sec:extendability-examples}, four examples illustrate these conditions and the non-extendable atoms that arise when a hypothesis fails.
In §\ref{sec:unicycle}, we give the auxiliary optimal-face test and apply progressive extraction and fixing to a unicycle trajectory problem. The numerical procedure recovers feasible, near-optimal trajectories, with the exact theoretical guarantees stated separately from the numerical tests.

\paragraph{Notation.}
For a positive integer $n$, let $[n]:=\{1,\ldots,n\}$. Let $x=(x_1,\ldots,x_n)$, and let $\mathbb R[x]$ be the set of real polynomials in $x$, with $\mathbb R[x]_d$ the subset of polynomials of degree at most $d$. For $\alpha\in\mathbb N^n$, write $x^\alpha=x_1^{\alpha_1}\cdots x_n^{\alpha_n}$ and $|\alpha|=\sum_i\alpha_i$. Let $\mathbb N_d^n=\{\alpha\in\mathbb N^n:|\alpha|\leq d\}$, and let $v_d(x)=(x^\alpha)_{\alpha\in\mathbb N_d^n}$ be the standard monomial basis. For $J\subseteq[n]$, let $x_J=(x_i)_{i\in J}$ be the corresponding subvector, $\pi_J$ the coordinate projection, and $\mathbb N_d^J=\{\alpha\in\mathbb N_d^n:\alpha_i=0\text{ for }i\notin J\}$. We write $A\succeq0$ if a symmetric matrix $A$ is positive semidefinite. For a set $K$, $\mathcal P(K)$ denotes the Borel probability measures supported on $K$, $\delta_u$ the Dirac measure at $u$, and $(\pi_J)_\#\mu$ the marginal of a measure $\mu$ on the coordinates $J$.

\section{Moment relaxations and solution extraction}
\label{sec:preliminaries}

In this section, we review how Lasserre's moment hierarchy relaxes a polynomial optimization problem (POP) into a semidefinite program (SDP), first in its dense form (§\ref{subsec:dense-moments}) and then with correlative sparsity (§\ref{subsec:sparse-moments}), together with the interpretation of each relaxation in terms of probability measures. We recall classical flat truncation and explain the atomic representation supplied by simplicial regularity (§\ref{subsec:flat-truncation}), which is the extraction condition used in Section~\ref{sec:global-extendability}. Throughout we work with the moment side of the hierarchy. Formally, we consider the POP
\begin{equation}
    f^\star=\min_{x\in K} f(x),\qquad
    K=\left\{x\in\mathbb R^n:
       g_j(x)\geq0\ (j=1,\ldots,m),\quad
       h_j(x)=0\ (j=1,\ldots,q)\right\},
    \label{eq:pop}
\end{equation}
where $f,g_j,h_j$ are real polynomials. We assume that $K$ is nonempty and compact, and write $X^\star=\{x\in K:f(x)=f^\star\}$ for its (nonempty) set of global minimizers.

\subsection{Dense moment relaxations}
\label{subsec:dense-moments}

\paragraph{Riesz functional and moment matrices.}
A truncated moment sequence $y=(y_\alpha)_{\alpha\in\mathbb N_{2\kappa}^n}$ defines the Riesz linear functional
\begin{equation}
    L_y(p)=\sum_\alpha p_\alpha y_\alpha
    \qquad\text{for } p(x)=\sum_\alpha p_\alpha x^\alpha\in\mathbb R[x]_{2\kappa}.
    \label{eq:moment-functional}
\end{equation}
$L_y$ replaces every monomial $x^\alpha$ by the number $y_\alpha$, so that polynomials become linear functions of $y$. Applied entrywise to polynomial matrices, it yields the moment and localizing matrices
\begin{equation}
    M_d(y)=L_y\bigl(v_dv_d^\top\bigr)
           =(y_{\alpha+\beta})_{\alpha,\beta\in\mathbb N_d^n},
    \qquad
    M_d(p\,y)=L_y\bigl(p\,v_dv_d^\top\bigr),
    \qquad 2d+\deg p\leq2\kappa.
    \label{eq:moment-localizing}
\end{equation}
For instance, for $n=3$, $M_1(y)$ is the $4\times4$ matrix indexed by $(1,x_1,x_2,x_3)$, and its entry in row $x_1$ and column $x_3$ is $y_{1,0,1}=L_y(x_1x_3)$. In general, $M_d(y)$ has size $\binom{n+d}{d}$. Put $d(p)=\lceil\deg(p)/2\rceil$, with $d(p)=0$ for constant polynomials, and define
\begin{equation}
    d_c=\max\bigl(\{1\}\cup\{d(g_j)\}_{j=1}^m
                    \cup\{d(h_j)\}_{j=1}^q\bigr),
    \qquad d_0=\max\{d_c,d(f)\}.
    \label{eq:dense-degree-orders}
\end{equation}

\begin{proposition}[Dense moment hierarchy, \citealt{lasserre2001}]
\label{prop:dense-hierarchy}
For $\kappa\geq d_0$, let $\rho_\kappa^{\mathrm{den}}$ denote the optimal value of
\begin{equation}
\begin{aligned}
    \rho_\kappa^{\mathrm{den}}=\inf_y\quad &L_y(f)\\
    \text{s.t.}\quad &y_0=1,\qquad M_\kappa(y)\succeq0,\\
    &M_{\kappa-d(g_j)}(g_jy)\succeq0,
       &&j=1,\ldots,m,\\
    &L_y(x^\alpha h_j)=0,
       &&|\alpha|\leq2\kappa-\deg h_j,\quad j=1,\ldots,q.
\end{aligned}
\label{eq:dense-relaxation}
\end{equation}
Then $\rho_\kappa^{\mathrm{den}}\leq\rho_{\kappa+1}^{\mathrm{den}}\leq f^\star$. If the constraint description of $K$ is Archimedean (Definition~\ref{def:archimedean}), then $\rho_\kappa^{\mathrm{den}}\uparrow f^\star$ as $\kappa\to\infty$.
\end{proposition}

\begin{definition}[Archimedean condition]
\label{def:archimedean}
Let $\Sigma[x]$ denote the cone of sums of squares of polynomials in $x\in\mathbb R^n$. For the constraint description in~\eqref{eq:pop}, the quadratic module generated by the inequalities together with the ideal generated by the equalities is
\begin{equation}
    Q(g,h)=\Bigl\{\sigma_0+\sum_{j=1}^m\sigma_jg_j+\sum_{j=1}^q\phi_jh_j:
       \ \sigma_0,\sigma_j\in\Sigma[x],\ \phi_j\in\mathbb R[x]\Bigr\}.
    \label{eq:quadratic-module}
\end{equation}
Every element of $Q(g,h)$ is nonnegative on $K$. The set $Q(g,h)$ is \emph{Archimedean} if $R-\|x\|^2\in Q(g,h)$ for some $R>0$.
\end{definition}

\paragraph{Measure interpretation.}
The POP~\eqref{eq:pop} is equivalent to the following optimization over probability measures on $K$~\citep{lasserre2001}:
\begin{equation}
    f^\star=\min_{\mu\in\mathcal P(K)}\int_K f(x)\,d\mu(x).
    \label{eq:measure-pop}
\end{equation}
Indeed, every such integral is at least $f^\star$, and the Dirac measure at a global minimizer attains this bound. Since $f-f^\star\geq0$ on $K$, a measure is optimal if and only if it is supported on $X^\star$. In particular, every atom $u_i$ of an optimal finitely atomic measure
\begin{equation}
    \mu=\sum_{i=1}^r\lambda_i\delta_{u_i},\qquad
    \lambda_i>0,\quad \sum_{i=1}^r\lambda_i=1,
    \label{eq:atomic-measure}
\end{equation}
is a global minimizer~\citep{yang2024moment}. The relaxation~\eqref{eq:dense-relaxation} keeps only the moments $y_\alpha=\int_Kx^\alpha\,d\mu$ with $|\alpha|\leq2\kappa$, and its constraints are necessary conditions for $y$ to have such a representing measure $\mu\in\mathcal P(K)$:
(i) $y_0=\mu(K)=1$;
(ii) $M_\kappa(y)=\int_Kv_\kappa v_\kappa^\top\,d\mu\succeq0$ and $M_{\kappa-d(g_j)}(g_jy)=\int_Kg_j\,v_{\kappa-d(g_j)}v_{\kappa-d(g_j)}^\top\,d\mu\succeq0$, since $g_j\geq0$ on $K$;
(iii) $L_y(x^\alpha h_j)=\int_Kx^\alpha h_j\,d\mu=0$, since $h_j=0$ on $K$.
As $L_y(f)=\int_Kf\,d\mu$, the moments of every $\mu\in\mathcal P(K)$ are feasible with the same objective value, therefore $\rho_\kappa^{\mathrm{den}}\leq f^\star$. 

\subsection{Sparse moment relaxations with correlative sparsity}
\label{subsec:sparse-moments}

\paragraph{Correlative sparsity.}
Many POPs involve interactions among only small subsets of variables. 
Following \citet{waki2005} and \citet{yang2024structure}, let $\mathcal C=\{\Delta_1,\ldots,\Delta_p\}$ be a cover of $[n]$ by overlapping \emph{variable cliques} such that the objective function decomposes as
\begin{equation}
    f(x)=\sum_{\ell=1}^p f_\ell(x_{\Delta_\ell}),
    \qquad f_\ell\in\mathbb R[x_{\Delta_\ell}],
    \label{eq:sparse-objective}
\end{equation}
and each constraint polynomial $g_j$ or $h_j$ belongs to $\mathbb R[x_{\Delta_\ell}]$ for at least one $\ell$. 
The cover satisfies the \emph{running intersection property} (RIP) if, after reordering, for every $\ell\geq2$ there is $s<\ell$ with
\begin{equation}
    \Delta_\ell\cap\left(\bigcup_{r<\ell}\Delta_r\right)\subseteq\Delta_s.
    \label{eq:rip}
\end{equation}

\begin{example}[Chain sparsity]
\label{ex:chain}
Consider a trajectory optimization problem over states $s_0,\ldots,s_N$ and controls $u_0,\ldots,u_{N-1}$ that minimizes $\sum_{k=1}^N\ell_k(s_{k-1},u_{k-1},s_k)$ subject to $s_0=s_{\mathrm{init}}$, the discretized dynamics $F_k(s_{k-1},u_{k-1},s_k)=0$, and path constraints $(s_{k-1},u_{k-1},s_k)\in C_k$ for $k\in[N]$, where $\ell_k$, $F_k$, and the inequalities describing $C_k$ are polynomial. 
Identifying each vector variable with the set of its coordinates, the cliques
\begin{equation}
    \Delta_k=\{s_{k-1},u_{k-1},s_k\},\qquad k\in[N],
    \label{eq:chain-cover}
\end{equation}
therefore form a valid cover, and it satisfies RIP in its natural order, since $\Delta_k\cap(\Delta_1\cup\cdots\cup\Delta_{k-1})=\{s_{k-1}\}\subseteq\Delta_{k-1}$.
\end{example}

Assign each constraint to one clique containing its variables, and let $G_\ell$ and $H_\ell$ index the inequalities and equalities assigned to $\Delta_\ell$. The local feasible sets are
\begin{equation}
    K_\ell=
    \left\{
        u\in\mathbb R^{|\Delta_\ell|}:
        g_j(u)\geq0\ (j\in G_\ell),\quad
        h_j(u)=0\ (j\in H_\ell)
    \right\},
    \label{eq:local-feasible-sets}
\end{equation}
so that $K=\{x\in\mathbb R^n:x_{\Delta_\ell}\in K_\ell\text{ for all }\ell\}$.

\paragraph{Sparse moment relaxation.}
The sparse hierarchy keeps only the moments of monomials supported in a single clique: its variable is $y=(y_\alpha)_{\alpha\in\mathcal A_{2\kappa}}$, where
\begin{equation}
    \mathcal A_{2\kappa}
        =\bigcup_{\ell=1}^p
          \mathbb N_{2\kappa}^{\Delta_\ell},
    \label{eq:sparse-moment-indices}
\end{equation}
which defines $L_y$ on $\sum_\ell\mathbb R[x_{\Delta_\ell}]_{2\kappa}$. For $J$ contained in a clique, let $M_d^J(y)=(y_{\alpha+\beta})_{\alpha,\beta\in\mathbb N_d^J}$, and let $M_d^J(g_jy)$ be the restriction of the dense localizing matrix to monomials in $x_J$ whenever $g_j\in\mathbb R[x_J]$. With $d(\cdot)$ as above, define
\begin{equation}
    d_\ell
        =\max\bigl(
            \{1\}\cup\{d(g_j):j\in G_\ell\}
            \cup\{d(h_j):j\in H_\ell\}
          \bigr),
    \qquad
    \kappa_0
        =\max_{1\leq\ell\leq p}
          \{d_\ell,d(f_\ell)\}.
    \label{eq:sparse-degree-orders}
\end{equation}
Let us formalize the hierarchy.

\begin{proposition}[Sparse moment hierarchy, \citealt{waki2005,lasserre2006sparse}]
\label{prop:sparse-hierarchy}
For $\kappa\geq\kappa_0$, let $\rho_\kappa$ denote the optimal value of
\begin{equation}
\begin{aligned}
    \rho_\kappa=\inf_y\quad
        &\sum_{\ell=1}^pL_y(f_\ell)\\
    \text{s.t.}\quad
        &y_0=1,\\
        &M_\kappa^{\Delta_\ell}(y)\succeq0,
            &&\ell=1,\ldots,p,\\
        &M_{\kappa-d(g_j)}^{\Delta_\ell}(g_jy)\succeq0,
            &&j\in G_\ell,\quad\ell=1,\ldots,p,\\
        &L_y(x^\alpha h_j)=0,
            &&\alpha\in\mathbb N_{2\kappa-\deg h_j}^{\Delta_\ell},
              \quad j\in H_\ell,\quad\ell=1,\ldots,p.
\end{aligned}
\label{eq:sparse-relaxation}
\end{equation}
Then $\rho_\kappa\leq\rho_{\kappa+1}$ and $\rho_\kappa\leq\rho_\kappa^{\mathrm{den}}\leq f^\star$. If the cover satisfies RIP and, for each $\ell$, the constraints assigned to $\Delta_\ell$ generate an Archimedean quadratic module in the variables $x_{\Delta_\ell}$ (Definition~\ref{def:archimedean}), then $\rho_\kappa\uparrow f^\star$ as $\kappa\to\infty$~\citep{lasserre2006sparse,yang2024structure}.
\end{proposition}
Since the clique moment and localizing matrices are principal submatrices of their dense counterparts, restricting a dense feasible sequence to $\mathcal A_{2\kappa}$ gives a sparse feasible one, which proves $\rho_\kappa\leq\rho_\kappa^{\mathrm{den}}$. We assume attainment whenever referring to an optimal moment solution, and call the relaxation \emph{value exact} when $\rho_\kappa=f^\star$.

\paragraph{Measure interpretation.}
Let $S_{\ell r}=\Delta_\ell\cap\Delta_r$. Every $\mu\in\mathcal P(K)$ induces local measures $\mu_\ell=(\pi_{\Delta_\ell})_\#\mu\in\mathcal P(K_\ell)$ whose marginals agree on the overlaps. Conversely, under RIP, every such compatible family is induced by a joint measure supported on $K$~\citep[Lemma 6.4]{lasserre2006sparse}. Hence~\eqref{eq:pop} is equivalent to
\begin{equation}
\begin{aligned}
    f^\star=\min_{\mu_1,\ldots,\mu_p}\quad
        &\sum_{\ell=1}^p
          \int_{K_\ell}f_\ell(u)\,d\mu_\ell(u)\\
    \text{s.t.}\quad
        &\mu_\ell\in\mathcal P(K_\ell),
          \qquad \ell=1,\ldots,p,\\
        &(\pi_{S_{\ell r}})_\#\mu_\ell
          =(\pi_{S_{\ell r}})_\#\mu_r,
          \qquad \ell<r,\quad S_{\ell r}\ne\varnothing.
\end{aligned}
\label{eq:sparse-measure-pop}
\end{equation}
The single measure on $K$ is replaced by one local measure per clique, and the marginal constraints glue the local measures together on the overlaps. 
The sparse relaxation keeps only the moments $y_\alpha=\int u^\alpha\,d\mu_\ell(u)$ with $\alpha\in\mathbb N_{2\kappa}^{\Delta_\ell}$. 
For each clique, its constraints are conditions (i)--(iii) of §\ref{subsec:dense-moments} for the clique subsequence $y|_{\Delta_\ell}=(y_\alpha)_{\alpha\in\mathbb N_{2\kappa}^{\Delta_\ell}}$ and the set $K_\ell$. 
The marginal constraints become the shared indexing, which enforces $\int u^\alpha\,d\mu_\ell=\int u^\alpha\,d\mu_r$ only for $\alpha\in\mathbb N_{2\kappa}^{S_{\ell r}}$.
\subsection{Atomic representations and solution extraction}
\label{subsec:flat-truncation}

We first recall flat truncation as classical background. It certifies that the truncation $y^\star|_{2t}=(y^\star_\alpha)_{|\alpha|\leq2t}$ of an optimal moment solution is the moment sequence of a finitely atomic measure on $K$, from which minimizers can be extracted. We then distinguish this rank condition from the simplicial regularity used for local extraction in the rest of the paper.

\begin{proposition}[Flat truncation, \citealt{nie2013flat}]
\label{prop:flat-truncation}
Let $y^\star$ be an optimal solution of~\eqref{eq:dense-relaxation}, and suppose that there is an integer $t$ with $d_0\leq t\leq\kappa$ such that
\begin{equation}
    \operatorname{rank}M_t(y^\star)
       =\operatorname{rank}M_{t-d_c}(y^\star)=r.
    \label{eq:dense-flat-truncation}
\end{equation}
Then $y^\star|_{2t}$ has an $r$-atomic representing measure~\eqref{eq:atomic-measure} on $K$, the relaxation is exact, and every atom is a global minimizer.
\end{proposition}

Indeed, condition~\eqref{eq:dense-flat-truncation} and the localizing constraints imply that $y^\star|_{2t}$ is represented by an $r$-atomic measure~\eqref{eq:atomic-measure} supported on $K$~\citep{nie2013flat}. Since $\deg f\leq2t$,
\begin{equation}
    \rho_\kappa^{\mathrm{den}}
      =L_{y^\star}(f)
      =\sum_{i=1}^r\lambda_i f(u_i)
      \geq f^\star
      \geq\rho_\kappa^{\mathrm{den}},
    \label{eq:dense-extraction-optimality}
\end{equation}
so all inequalities hold with equality, and $f(u_i)=f^\star$ for every $i$ since $\lambda_i>0$. In words, flat truncation means that raising the degree of the moment matrix from $t-d_c$ to $t$ reveals no new rank: the optimal moments are then those of an optimal measure~\eqref{eq:atomic-measure}.

\paragraph{Classical sparse extraction.}
The same reasoning applies clique by clique. Let $y$ be feasible for~\eqref{eq:sparse-relaxation}, and let $d_\ell\leq t\leq\kappa$. If
\begin{equation}
    \operatorname{rank}M_t^{\Delta_\ell}(y)
       =\operatorname{rank}M_{t-d_\ell}^{\Delta_\ell}(y)=r_\ell,
    \label{eq:local-flat-truncation}
\end{equation}
then the argument above, applied to $y|_{\Delta_\ell}$ and the constraints assigned to $\Delta_\ell$, shows that the clique moments through degree $2t$ are those of an atomic measure
\begin{equation}
    \mu_\ell=\sum_{i=1}^{r_\ell}\lambda_{\ell i}\delta_{a_{\ell i}},
    \qquad a_{\ell i}\in K_\ell,\quad
    \lambda_{\ell i}>0,\quad\sum_i\lambda_{\ell i}=1.
    \label{eq:local-atomic-measure}
\end{equation}
However, the atoms of different cliques need not agree on their shared variables. The classical sufficient condition solves this by pinning every overlap to a single value.

\begin{proposition}[Sparse extraction, cf.~{\citealt[Theorem 3.7]{lasserre2006sparse}}]
\label{prop:sparse-extraction}
Let $y^\star$ be an optimal solution of~\eqref{eq:sparse-relaxation}, and suppose that, at a common order $t$ with $\kappa_0\leq t\leq\kappa$, every clique satisfies~\eqref{eq:local-flat-truncation} and
\begin{equation}
    \operatorname{rank}M_t^{S_{\ell r}}(y^\star)=1
    \qquad\text{for every nonempty overlap }S_{\ell r}.
    \label{eq:sparse-overlap-rank-one}
\end{equation}
Then $\rho_\kappa=f^\star$, and any choice of one atom $a_{\ell i_\ell}$ of $\mu_\ell$ for each $\ell$ assembles into a global minimizer $x^\star$ with $x^\star_{\Delta_\ell}=a_{\ell i_\ell}$ for all $\ell$.
\end{proposition}

Proposition~\ref{prop:sparse-extraction} requires every clique to be flat. It does not address what can be inferred from an individual clique when the remaining blocks fail this condition.

\begin{definition}[Partial flatness]
\label{def:partial-flatness}
Let $y$ be feasible for~\eqref{eq:sparse-relaxation} and let $d_\ell\leq t\leq\kappa$. We say that $y$ is \emph{flat on} $\Delta_\ell$ at order $t$ if the flat truncation condition~\eqref{eq:local-flat-truncation} holds for this clique, regardless of the other cliques. We say that $y$ is \emph{partially flat} if it is flat on some, but not necessarily all, of the cliques.
\end{definition}

If $y$ is flat on $\Delta_\ell$ at order $t$, then the moments of $y$ on $\Delta_\ell$ through degree $2t$ are those of an atomic measure~\eqref{eq:local-atomic-measure} supported on $K_\ell$.

\paragraph{Atomic extraction under simplicial regularity.}
Our local extraction condition is instead simplicial regularity, whose precise facial definition is given in Definition~\ref{def:simplicial-face}. For a normalized block $M_\kappa^{\Delta_\ell}(y)$ satisfying this condition, the minimal face in the pseudo-moment matrix cone is generated by linearly independent finite point-evaluation matrices. Consequently, there are distinct atomic candidates $a_{\ell i}$ and unique weights $\lambda_{\ell i}>0$, with $\sum_i\lambda_{\ell i}=1$, such that
\begin{equation}
    y_\alpha=\sum_i\lambda_{\ell i}a_{\ell i}^\alpha,
    \qquad \alpha\in\mathbb N_{2\kappa}^{\Delta_\ell}.
    \label{eq:local-simplicial-decomposition}
\end{equation}
This gives a full-degree atomic decomposition without a flatness assumption~\citep{kang2026simplicial}. Local feasibility must be checked separately: if every $a_{\ell i}\in K_\ell$, the resulting measure has the form~\eqref{eq:local-atomic-measure}. Simpliciality alone, without generation by finite point evaluations, is not sufficient. For an optimal dense moment solution, a decomposition with all candidates in $K$ gives global minimizers by the objective argument~\eqref{eq:dense-extraction-optimality}. Likewise, for an optimal sparse moment solution, the assembly conclusion of Proposition~\ref{prop:sparse-extraction} holds when every clique has a full-degree locally feasible atomic representation and every overlap block has rank one: zero overlap variances force agreement of the atoms, and averaging the assembled objective values gives $\rho_\kappa$. A locally feasible decomposition on one clique, however, does not by itself certify global extendability.

\begin{definition}[Locally feasible atomic representation]
\label{def:local-representation}
A measure of the form~\eqref{eq:local-atomic-measure}, with $r_\ell$ denoting its number of atoms, that represents the moments of $y$ on $\Delta_\ell$ through degree $2t$ is a \emph{locally feasible atomic representation} of $y$ on $\Delta_\ell$. Its atoms $a_{\ell i}$ are the \emph{extracted atoms}.
\end{definition}

In Section~\ref{sec:global-extendability} and the applications of its certificates, we use this definition with $t=\kappa$: the measure must represent every clique moment through degree $2\kappa$. Simplicial regularity supplies this representation when all atomic candidates are locally feasible. Discarding an infeasible candidate generally destroys the moment representation. The remaining question is when a candidate, or every atom of a locally feasible representation, extends to a global minimizer.

\section{From a motivating example}
\label{sec:partial-flatness}

The classical sparse extraction criterion in Proposition~\ref{prop:sparse-extraction} requires all clique moment matrices to be flat.
We instead focus on atomic extraction from selected blocks under simplicial regularity (Definition~\ref{def:simplicial-face}).
The following example shows that a single such block can identify part of a global minimizer, even when the remaining blocks are not simplicially regular. It raises the question: under what conditions can an atomic candidate extracted from an individual simplicially regular clique be extended to a global minimizer?

\begin{example}
\label{ex:chain-partial-flatness}
Consider the polynomial optimization problem
\begin{equation}
\begin{aligned}
    f^\star=\min_{x\in[-1,1]^4}\ f(x),\qquad
      &f(x)=f_1(x_1,x_2)+f_2(x_2,x_3)+f_3(x_3,x_4),\\
    f_1=(x_1x_2)^2,\qquad
      &f_2=x_2^2+(x_2-x_3)^2,\qquad
      f_3=(x_3x_4)^2,
\end{aligned}
\label{eq:chain-pop}
\end{equation}
with the cliques
\begin{equation}
    \Delta_1=\{x_1,x_2\},\qquad
    \Delta_2=\{x_2,x_3\},\qquad
    \Delta_3=\{x_3,x_4\}.
    \label{eq:chain-cliques}
\end{equation}
Assign each box constraint $1-x_j^2\geq0$ to every clique containing $x_j$. The cover is a chain and satisfies RIP: its only nonempty overlaps are $S_{12}=\{x_2\}$ and $S_{23}=\{x_3\}$. 
All constraints have degree two, so $d_\ell=1$, and the quartic objective gives $\kappa_0=2$.
\end{example}

The objective is a sum of squares. It vanishes if and only if $x_2=x_3=0$, while $x_1$ and $x_4$ can be chosen arbitrarily in $[-1,1]$. Thus
\begin{equation}
    f^\star=0,\qquad
    X^\star=\{(s,0,0,t):s,t\in[-1,1]\}.
    \label{eq:chain-minimizers}
\end{equation}
At relaxation order $\kappa=2$, let $U,V$ be independent and uniform on $[-1,1]$, and consider the normalized functional
\begin{equation}
    L(p)=\mathbb E\bigl[p(U,0,0,V)\bigr]
        =\frac14\int_{-1}^1\!\!\int_{-1}^1p(s,0,0,t)\,ds\,dt,
    \qquad p\in\sum_{\ell=1}^3\mathbb R[x_{\Delta_\ell}]_4.
    \label{eq:chain-functional}
\end{equation}
Write $\nu$ for the uniform probability measure on $[-1,1]$. The clique representing measures of $L$ are
\begin{equation}
    \mu_1=\nu\otimes\delta_0,\qquad
    \mu_2=\delta_{(0,0)},\qquad
    \mu_3=\delta_0\otimes\nu.
    \label{eq:chain-local-measures}
\end{equation}

\begin{proposition}
\label{prop:chain-partial-flatness}
For Example~\ref{ex:chain-partial-flatness} and $\kappa=2$, $L$ is an optimal solution of the sparse relaxation and $\rho_2=L(f)=f^\star=0$. Its block on $\Delta_2$ satisfies simplicial regularity, with
\[
    \operatorname{rank}M_2^{\Delta_2}(L)=\operatorname{rank}M_1^{\Delta_2}(L)=1,
\]
and its unique atomic candidate is the locally feasible point $a=(a_2,a_3)=(0,0)$. The other two blocks are not simplicially regular; for $\ell\in\{1,3\}$, their ranks are $\operatorname{rank}M_1^{\Delta_\ell}(L)=2$ and $\operatorname{rank}M_2^{\Delta_\ell}(L)=3$. The atom $a$ extends to a global minimizer.
\end{proposition}

\begin{proof}

\emph{Feasibility and optimality.} The functional $L$ integrates a probability measure supported on $X^\star\subseteq K$, so its restriction to the sparse moment indices is feasible for~\eqref{eq:sparse-relaxation}. In particular, all moment and localizing matrices are positive semidefinite, and the overlap moments agree. For any feasible normalized functional $\widetilde L$, each squared polynomial in the objective has degree at most two before squaring and belongs to one clique. Positive semidefiniteness of its order-two moment block therefore gives
\begin{equation}
\begin{aligned}
    \widetilde L(f)
      &=\widetilde L\bigl((x_1x_2)^2\bigr)+\widetilde L(x_2^2)
        +\widetilde L\bigl((x_2-x_3)^2\bigr)+\widetilde L\bigl((x_3x_4)^2\bigr)\\
      &\geq0=L(f)=f^\star.
\end{aligned}
\label{eq:chain-lower-bound}
\end{equation}
Hence $\rho_2=f^\star=0$.

\emph{Simplicial regularity and extension to global minimizers.} Since $\mu_2=\delta_{(0,0)}$, $M_2^{\Delta_2}(L)=A_2^{\Delta_2}(0,0)$ is a normalized rank-one point-evaluation matrix. Its minimal face in $\mathcal P_2^{\Delta_2}$ is the ray generated by this matrix, so it satisfies simplicial regularity and $a=(0,0)$ is its unique candidate. This candidate satisfies all assigned box constraints and gives the full-degree locally feasible representation $\delta_{(0,0)}$.
For either endpoint block, the restriction to the univariate monomials $(1,s,s^2)$, with $s=x_1$ or $x_4$, is the positive definite moment matrix of the uniform measure on $[-1,1]$. Its range therefore contains every evaluation vector on the corresponding coordinate line. Subtracting a sufficiently small positive multiple of any such evaluation matrix preserves positive semidefiniteness and the moment structure, so the minimal face contains infinitely many distinct point-evaluation rays. These rank-one rays are extreme, which rules out a simplicial minimal face.
By~\eqref{eq:chain-minimizers},
\begin{equation}
    a=(0,0)\in\pi_{\Delta_2}(X^\star)=\{(0,0)\},
    \qquad
    (s,a_2,a_3,t)\in X^\star\quad\text{for all }s,t\in[-1,1].
    \label{eq:chain-extension}
\end{equation}
Thus fixing $x_{\Delta_2}=a$ preserves every global minimizer.
\end{proof}

At the above optimal solution, Proposition~\ref{prop:sparse-extraction} is inapplicable because the endpoint blocks are not flat. Nevertheless, the single simplicially regular clique already identifies coordinates shared by every global minimizer. Therefore, we ask the following question:
\begin{quote}\itshape
When only selected clique blocks satisfy simplicial regularity, under what conditions does an extracted atomic candidate extend to a global minimizer?
\end{quote}
We will try to address this question in the next section.

\section{Sufficient conditions for global extendability}
\label{sec:global-extendability}

In this section, we will develop an existence guarantee from simplicial regularity and maximum rank over the optimal face; rank-one extendability follows as a corollary.

We use the sparse relaxation and notation of Section~\ref{sec:preliminaries}. The objective has a fixed decomposition $f=\sum_{\Gamma\in\mathcal C}f_\Gamma$, every constraint is assigned to a clique containing its variables, and distinct cliques are coupled only through shared moment entries. 
The feasible set $K$ is nonempty and compact. 
Fix $\kappa\geq\kappa_0$ and assume that the sparse relaxation has a finite, attained optimum $\rho_\kappa$. Value exactness is imposed only where stated.

\subsection{Global extendability and clique boundaries}
\label{subsec:extendability}

A locally feasible atomic representation on a clique certifies that its atoms satisfy that clique's assigned constraints. To fix an atom without excluding every global minimizer, we need the following additional property.

\begin{definition}[Global extendability]
\label{def:extendability}
For $J\subseteq[n]$, an assignment $c\in\mathbb R^{|J|}$ is \emph{globally extendable} if $c\in\pi_J(X^\star)$, that is, if some global minimizer agrees with $c$ on $J$. In particular, an extracted atom $a$ on a clique $\Delta$ is globally extendable if $a\in\pi_\Delta(X^\star)$.
\end{definition}

Thus fixing $x_J=c$ retains at least one global minimizer exactly when $c$ is globally extendable.

\begin{definition}[Boundary and interior variables]
\label{def:boundary}
For a selected clique $\Delta\in\mathcal C$, define its \emph{boundary} and \emph{interior} by
\begin{equation}
    B_\Delta=\Delta\cap
         \Bigl(\bigcup_{\Gamma\in\mathcal C,\,\Gamma\ne\Delta}\Gamma\Bigr),
    \qquad I_\Delta=\Delta\setminus B_\Delta.
    \label{eq:clique-boundary}
\end{equation}
\end{definition}

Write $x_\Delta=(z,b)$, where $z=x_{I_\Delta}$ and $b=x_{B_\Delta}$, and write an atom as $a_i=(z_i,b_i)$. The boundary includes every variable shared with another clique and is the only interface between $\Delta$ and the remaining problem. Define the local fiber
\begin{equation}
    K_\Delta(b)=\{z:(z,b)\in K_\Delta\}.
    \label{eq:conditional-clique-problem}
\end{equation}
Once $b$ is fixed, the selected clique can optimize its interior by minimizing $f_\Delta(\cdot,b)$ over $K_\Delta(b)$.

\subsection{Extendability from simplicial regularity over the optimal face}
\label{subsec:rank-one-extendability}

Let $\mathcal S_\kappa$ be the feasible set of normalized sparse moment sequences and define
\begin{equation}
    \mathcal F_\kappa
       =\{y\in\mathcal S_\kappa:L_y(f)=\rho_\kappa\}.
    \label{eq:optimal-face}
\end{equation}
For $J$ contained in a clique, the matrix $M_\kappa^J(y)$ is affine and positive semidefinite on $\mathcal F_\kappa$. Its maximum rank captures the ranges of all optimal blocks.

\begin{lemma}[Common range of optimal moment blocks]
\label{lem:common-range}
Let $J$ be contained in a clique, and let $y^\circ\in\mathcal F_\kappa$ satisfy
\begin{equation}
    \operatorname{rank}M_\kappa^J(y^\circ)
      =\max_{y\in\mathcal F_\kappa}
         \operatorname{rank}M_\kappa^J(y).
    \label{eq:maximum-block-rank}
\end{equation}
Then, for every $y\in\mathcal F_\kappa$,
\begin{equation}
    \operatorname{range}M_\kappa^J(y)
       \subseteq\operatorname{range}M_\kappa^J(y^\circ),
    \qquad
    \ker M_\kappa^J(y^\circ)\subseteq\ker M_\kappa^J(y).
    \label{eq:common-range}
\end{equation}
\end{lemma}

\begin{proof}
Put $A=M_\kappa^J(y^\circ)$ and $B=M_\kappa^J(y)$. Convexity gives $(y^\circ+y)/2\in\mathcal F_\kappa$, with block $(A+B)/2$. Since $A,B\succeq0$, $\ker(A+B)=\ker A\cap\ker B$. Thus $\operatorname{rank}(A+B)\geq\operatorname{rank}A$, while maximum rank gives the reverse inequality. Hence $\ker A\subseteq\ker B$; taking orthogonal complements gives the range inclusion.
\end{proof}

The maximum in~\eqref{eq:maximum-block-rank} is attained because $\mathcal F_\kappa$ is nonempty and the possible ranks form a finite set of integers. The same argument applies to $M_t^J$ for any $t\leq\kappa$.

\paragraph{Simplicial regularity and atomic decomposition.}
For $J$ contained in a clique, define the pseudo-moment matrix cone and the point-evaluation matrix by
\begin{equation}
    \mathcal P_\kappa^J
      =\bigl\{M_\kappa^J(w):
          w=(w_\alpha)_{\alpha\in\mathbb N_{2\kappa}^J},\
          M_\kappa^J(w)\succeq0\bigr\},
    \qquad
    A_\kappa^J(a)=v_\kappa^J(a)\,v_\kappa^J(a)^\top.
    \label{eq:pseudo-moment-cone}
\end{equation}
Here $v_\kappa^J$ lists the monomials of degree at most $\kappa$ in $x_J$. The cone imposes the moment matrix structure and positive semidefiniteness, without normalization or localizing constraints. After homogenization, it is the matrix realization of the homogeneous pseudo-moment cone used by \citet{kang2026simplicial}. Their Theorem~1 gives a generic regime in which the following facial property holds. We use the term \emph{simplicial regularity} for this property of the selected block and assume it directly.

\begin{definition}[Simplicial regularity]
\label{def:simplicial-face}
A matrix $A\in\mathcal P_\kappa^J$ satisfies \emph{simplicial regularity} if there are distinct finite points $a_1,\ldots,a_r\in\mathbb R^{|J|}$ such that
\begin{equation}
    \operatorname{minface}(A,\mathcal P_\kappa^J)
      =\operatorname{cone}\{A_\kappa^J(a_1),\ldots,A_\kappa^J(a_r)\},
    \label{eq:simplicial-face}
\end{equation}
with the $r$ generators linearly independent. We call the $a_i$ \emph{atomic candidates}.
\end{definition}

The condition requires both simpliciality and generation by finite point-evaluation rays. Simpliciality alone does not exclude higher-rank generators. Simplicial regularity gives an atomic decomposition of the selected block without a flatness assumption. When $J=\Delta$, the candidates need not all satisfy the assigned clique constraints; their local feasibility can be checked after extraction.

\begin{theorem}[Global extendability under simplicial regularity]
\label{thm:simplicial-face}
Assume $\rho_\kappa=f^\star$. Let $J$ be contained in a clique, and let $\bar y\in\mathcal F_\kappa$ have a block $M_\kappa^J(\bar y)$ of maximum rank over $\mathcal F_\kappa$. Suppose $A=M_\kappa^J(\bar y)$ satisfies simplicial regularity as in Definition~\ref{def:simplicial-face}, with atomic candidates $a_1,\ldots,a_r$. Then there are unique coefficients $\lambda_i>0$, with $\sum_i\lambda_i=1$, such that
\begin{equation}
    \bar y_\alpha=\sum_{i=1}^r\lambda_ia_i^\alpha
    \qquad(\alpha\in\mathbb N_{2\kappa}^J).
    \label{eq:simplicial-decomposition}
\end{equation}
Every optimal block $M_\kappa^J(y)$ belongs to $\operatorname{conv}\{A_\kappa^J(a_1),\ldots,A_\kappa^J(a_r)\}$, and
\begin{equation}
    \varnothing\ne\pi_J(X^\star)\subseteq\{a_1,\ldots,a_r\}.
    \label{eq:simplicial-projection}
\end{equation}
In particular, at least one atomic candidate is globally extendable. If $J=\Delta$ is a clique, then $\pi_\Delta(X^\star)\subseteq\{a_1,\ldots,a_r\}\cap K_\Delta$, so at least one candidate is also locally feasible. If every $a_i\in K_\Delta$, the measure $\sum_i\lambda_i\delta_{a_i}$ is a locally feasible atomic representation in the sense of Definition~\ref{def:local-representation}.
\end{theorem}

\begin{proof}
Write $F=\operatorname{minface}(A,\mathcal P_\kappa^J)$. A point lies in the relative interior of its minimal face, so simpliciality gives a unique decomposition $A=\sum_i\lambda_iA_\kappa^J(a_i)$ with all $\lambda_i>0$. The constant entry gives $\sum_i\lambda_i=1$. Equality of the moment matrices determines every moment through degree $2\kappa$, proving~\eqref{eq:simplicial-decomposition}.

Let $y\in\mathcal F_\kappa$ and $B=M_\kappa^J(y)$. Lemma~\ref{lem:common-range} gives $\operatorname{range}B\subseteq\operatorname{range}A$. Thus $A-\varepsilon B\succeq0$ for sufficiently small $\varepsilon>0$. Both matrices have the moment matrix structure, so $A-\varepsilon B\in\mathcal P_\kappa^J$. The identity $A=(A-\varepsilon B)+\varepsilon B$ and the face property imply $B\in F$. By normalization, $B$ is a convex combination of the atomic generators.

Now fix any $x^\star\in X^\star$. Value exactness places its evaluation sequence in $\mathcal F_\kappa$, so its block satisfies
\[
    A_\kappa^J(x^\star_J)=\sum_{i=1}^r\beta_iA_\kappa^J(a_i),
    \qquad\beta_i\geq0,\quad\sum_i\beta_i=1.
\]
The matrix on the left has rank one. Every positive semidefinite summand with positive coefficient must have its range in that same one-dimensional space. Since all evaluation vectors have constant coordinate one, any such generator equals $A_\kappa^J(x^\star_J)$. The first moments, with $\kappa\geq1$, then give $a_i=x^\star_J$. This proves~\eqref{eq:simplicial-projection}; nonemptiness follows from the standing assumptions. If $J=\Delta$, global feasibility implies $x^\star_\Delta\in K_\Delta$, giving the final claims.
\end{proof}

The rank-one case forces all optimal blocks to agree and gives a unique globally extendable assignment.

\begin{corollary}[Rank-one extendability]
\label{cor:rank-one-extendability}
Assume $\rho_\kappa=f^\star$. Let $J$ be contained in a clique and suppose
\begin{equation}
    \max_{y\in\mathcal F_\kappa}
       \operatorname{rank}M_\kappa^J(y)=1.
    \label{eq:rank-one-optimal-face}
\end{equation}
Then all $y\in\mathcal F_\kappa$ have the same block $M_\kappa^J(y)$, and there is a unique $c\in\mathbb R^{|J|}$ such that
\begin{equation}
    \pi_J(X^\star)=\{c\}.
    \label{eq:unique-extendable-value}
\end{equation}
Its coordinates are $c_j=L_y(x_j)$ for any $y\in\mathcal F_\kappa$ and $j\in J$. Thus $c$ is the unique globally extendable assignment on $J$.
\end{corollary}

\begin{proof}
Choose $\bar y\in\mathcal F_\kappa$. Its normalized rank-one block has the form $A_\kappa^J(c)$: factor it as $ww^\top$ with $w_0=1$, put $c_j=w_{e_j}$, and use the Hankel identities to obtain $w_\alpha=c^\alpha$. Its minimal face in $\mathcal P_\kappa^J$ is the ray generated by $A_\kappa^J(c)$. Apply Theorem~\ref{thm:simplicial-face} with one generator. Every optimal block is that generator, and the nonempty set $\pi_J(X^\star)$ is contained in $\{c\}$.
\end{proof}

These are existence guarantees. The rank-two ring example in Section~\ref{sec:extendability-examples} has a block of maximum rank two with a simplicial minimal face generated by the candidates $(0,0,0)$ and $(1,1,1)$; only the latter is globally extendable. Conversely, the value-exact counterexample with a non-extendable rank-one clique shows why maximum rank cannot simply be omitted. Its selected block satisfies simplicial regularity, but combining its optimal sequence with the evaluation sequence of a global minimizer increases the block rank.

\subsection{Extendability through the clique boundary}
\label{subsec:boundary-certificate}

The preceding results ensure that the candidate set contains a globally extendable assignment. To certify every atom of a selected clique, we combine conditional optimality with a boundary value shared by all global minimizers.

Fix $\Delta\in\mathcal C$ and write $x_\Delta=(z,b)$ as in §\ref{subsec:extendability}. Let $\bar y\in\mathcal F_\kappa$ admit a locally feasible atomic representation through the full relaxation degree:
\begin{equation}
    \mu_\Delta=\sum_{i=1}^r\lambda_i\delta_{(z_i,b_i)},
    \qquad (z_i,b_i)\in K_\Delta,\quad
    \lambda_i>0,\quad \sum_i\lambda_i=1,
    \label{eq:theory-local-measure}
\end{equation}
with
\begin{equation}
    L_{\bar y}(q)=\int q\,d\mu_\Delta
      \qquad\text{for every }q\in\mathbb R[x_\Delta]_{2\kappa}.
    \label{eq:theory-full-representation}
\end{equation}
This is the full-degree instance of Definition~\ref{def:local-representation}. In particular, simplicial regularity in Theorem~\ref{thm:simplicial-face} supplies this representation when $J=\Delta$ and all of its atomic candidates satisfy the clique constraints.

\begin{lemma}[Atomic replacement]
\label{lem:atomic-replacement}
Let $\bar y\in\mathcal F_\kappa$ satisfy~\eqref{eq:theory-local-measure}--\eqref{eq:theory-full-representation}. Then every atom is conditionally optimal over the interior variables:
\begin{equation}
    z_i\in\operatorname*{arg\,min}_{z\in K_\Delta(b_i)}
                  f_\Delta(z,b_i),
    \qquad i=1,\ldots,r.
    \label{eq:conditional-atomic-optimality}
\end{equation}
\end{lemma}

\begin{proof}
Suppose that $\widetilde z\in K_\Delta(b_i)$ satisfies $f_\Delta(\widetilde z,b_i)<f_\Delta(z_i,b_i)$. Move the atom $(z_i,b_i)$ to $(\widetilde z,b_i)$, obtaining $\widetilde\mu_\Delta$. Define $\widetilde y$ by using its moments on $\Delta$ and the moments of $\bar y$ elsewhere.

Every moment shared with another clique is supported on $B_\Delta$, whose marginal is unchanged. Thus the remaining clique blocks and their constraints are unchanged. On $\Delta$, the new moments come from a probability measure supported on $K_\Delta$, so all assigned constraints hold. Hence $\widetilde y$ is feasible, but
\begin{equation}
    L_{\widetilde y}(f)
       =\rho_\kappa
         +\lambda_i\bigl(
            f_\Delta(\widetilde z,b_i)-f_\Delta(z_i,b_i)
          \bigr)
       <\rho_\kappa,
    \label{eq:replacement-objective-change}
\end{equation}
contradicting optimality.
\end{proof}

Combining Corollary~\ref{cor:rank-one-extendability} on the boundary with Lemma~\ref{lem:atomic-replacement} gives a certificate for every extracted atom, without requiring the entire clique block to have rank one.

\begin{theorem}[Boundary rank-one certificate]
\label{thm:boundary-certificate}
Assume $\rho_\kappa=f^\star$. Let $\bar y\in\mathcal F_\kappa$ admit the locally feasible representation~\eqref{eq:theory-local-measure}--\eqref{eq:theory-full-representation}. If
\begin{equation}
    \max_{y\in\mathcal F_\kappa}
       \operatorname{rank}M_\kappa^{B_\Delta}(y)=1,
    \label{eq:boundary-rank-one-condition}
\end{equation}
then every atom $(z_i,b_i)$ of $\mu_\Delta$ is globally extendable.
\end{theorem}

\begin{proof}
Apply Corollary~\ref{cor:rank-one-extendability} with $J=B_\Delta$. There is a value $c$ shared by all global minimizers and all optimal boundary blocks. For every $j\in B_\Delta$, the common rank-one block gives $L_{\bar y}(x_j)=c_j$ and $L_{\bar y}(x_j^2)=c_j^2$. Therefore
\begin{equation}
    \sum_{i=1}^r\lambda_i(b_{i,j}-c_j)^2
       =L_{\bar y}\bigl((x_j-c_j)^2\bigr)=0.
    \label{eq:boundary-zero-variance}
\end{equation}
Since every weight is positive, $b_i=c$ for all $i$.

Fix an atom $(z_i,c)$ and any $x^\star\in X^\star$. Write $z^\star=x^\star_{I_\Delta}$ and replace its interior coordinates by $z_i$:
\[
    \widetilde x_{I_\Delta}=z_i,
    \qquad
    \widetilde x_{[n]\setminus I_\Delta}
       =x^\star_{[n]\setminus I_\Delta}.
\]
The selected clique is feasible because $(z_i,c)\in K_\Delta$. All other cliques are unchanged, since none contains an interior variable of $\Delta$. Thus $\widetilde x\in K$. Lemma~\ref{lem:atomic-replacement} and $z^\star\in K_\Delta(c)$ give
\begin{equation}
    f(\widetilde x)
       =f(x^\star)+f_\Delta(z_i,c)-f_\Delta(z^\star,c)
       \leq f^\star.
    \label{eq:certificate-objective}
\end{equation}
Feasibility forces equality. Hence $\widetilde x\in X^\star$ and $\widetilde x_\Delta=(z_i,b_i)$, proving global extendability.
\end{proof}
Corollary~\ref{cor:rank-one-extendability} and Theorem~\ref{thm:boundary-certificate} derive global extendability from value exactness and rank one of a moment block over the whole optimal face, with a full-degree locally feasible representation required for the boundary certificate. 
Theorem~\ref{thm:simplicial-face} guarantees at least one globally extendable candidate under simplicial regularity and maximum rank over that face. In the next section, we will illustrate these results with several examples, and we will discuss that value exactness and rank one of a moment block cannot be dropped for global extendability.

\section{Examples for the extendability conditions}
\label{sec:extendability-examples}

We illustrate the certificates of Section~\ref{sec:global-extendability} with four examples. The first two show when the certificates apply; the last two show how extendability can fail without value exactness or the boundary rank-one condition. All relaxations have order $\kappa=2$.

We first illustrate the extendability guarantees on a value-exact ring. The selected clique has rank one over the whole optimal face, so its unique atom is globally extendable although the remaining blocks are not simplicially regular.

\begin{example}[Exact ring with a rank-one clique]
\label{ex:ring-rank-one}
Consider
\begin{equation}
\begin{aligned}
    f^\star=\min_{(x_1,x_2,x_3,x_4,z)\in[-1,1]^5}\quad
      &f_1(x_1,x_2,z)+f_2(x_2,x_3)+f_3(x_3,x_4)+f_4(x_4,x_1),\\
    f_1=x_1^2+x_2^2+(z-x_1-x_2)^2,\quad
      &f_2=(x_2x_3)^2,\quad f_3=(x_3-x_4)^2,\quad f_4=(x_4x_1)^2,
\end{aligned}
\label{eq:ring-pop}
\end{equation}
with the cliques $\Delta_1=\{x_1,x_2,z\}$, $\Delta_2=\{x_2,x_3\}$, $\Delta_3=\{x_3,x_4\}$, and $\Delta_4=\{x_4,x_1\}$. Assign each box constraint $1-s^2\geq0$ to every clique containing $s$. The cover violates RIP: whichever clique is ordered last shares a different variable with each neighbor, and no other clique contains both. For the selected clique $\Delta_1$, the boundary is $B_{\Delta_1}=\{x_1,x_2\}$ and the interior is $I_{\Delta_1}=\{z\}$.
\end{example}

\emph{Global problem.} The objective is a sum of squares, all of which vanish exactly when $x_1=x_2=z=0$ and $x_3=x_4$. Hence
\begin{equation}
    f^\star=0,\qquad
    X^\star=\{(0,0,t,t,0):t\in[-1,1]\}.
    \label{eq:ring-minimizers}
\end{equation}

\emph{Relaxation.} Each summand is a square of a polynomial of degree at most two supported on a clique, so $L(f)\geq0$ for every feasible functional $L$. The moments of $(0,0,T,T,0)$, with $T$ uniform on $[-1,1]$, are feasible and attain zero. Thus $\rho_2=f^\star=0$. The selected block is $M_2^{\Delta_1}(L)=A_2^{\Delta_1}(0,0,0)$, whose minimal face is a single point-evaluation ray. It satisfies simplicial regularity and has the full-degree locally feasible representation $\delta_{(0,0,0)}$.

To check the rank over the whole optimal face $\mathcal F_2$, take any optimal $L$. The three squares in $f_1$ have value zero, giving $L(s^2)=0$ for $s\in\{x_1,x_2,z\}$. The box localizers give $0\leq L(s^4)\leq L(s^2)=0$, so positive semidefiniteness annihilates the rows of $s$ and $s^2$. The latter row then gives $L(s^2t^2)=0$, annihilating the row of $st$ for every other clique variable $t$. Therefore $M_2^{\Delta_1}(L)=A_2^{\Delta_1}(0,0,0)$ for every $L\in\mathcal F_2$: both the selected block and its boundary have maximum rank one.

At the displayed solution, the other clique measures are the laws of $(0,T)$, $(T,T)$, and $(T,0)$. Their blocks have ranks two and three at orders one and two, respectively. As in Example~\ref{ex:chain-partial-flatness}, their minimal faces contain infinitely many point-evaluation rays, so these blocks are not simplicially regular.

\emph{Extendability.} Corollary~\ref{cor:rank-one-extendability}, the rank-one case of Theorem~\ref{thm:simplicial-face}, gives $\pi_{\Delta_1}(X^\star)=\{(0,0,0)\}$. Theorem~\ref{thm:boundary-certificate} also applies to $\delta_{(0,0,0)}$. Thus the unique extracted atom agrees with every global minimizer.

The next example shows that Theorem~\ref{thm:boundary-certificate} requires rank one only on the boundary. A selected clique of rank two can have multiple extracted atoms, all globally extendable.

\begin{example}[Rank-one boundary with two atoms]
\label{ex:ring-boundary}
In Example~\ref{ex:ring-rank-one}, replace $f_1$ by
\begin{equation}
    \widetilde f_1(x_1,x_2,z)=x_1^2+x_2^2+\Bigl(z^2-\frac14-x_1-x_2\Bigr)^2,
    \label{eq:ring-boundary-objective}
\end{equation}
keeping the constraints, their assignments, the cover, and $f_2$, $f_3$, $f_4$.
\end{example}

\emph{Global problem.} The objective $\widetilde f=\widetilde f_1+f_2+f_3+f_4$ is a sum of squares, all of which vanish exactly when $x_1=x_2=0$, $z^2=1/4$, and $x_3=x_4$. Hence
\begin{equation}
    \widetilde f^\star=0,\qquad
    \widetilde X^\star=\{(0,0,t,t,\pm\tfrac12):t\in[-1,1]\}.
    \label{eq:ring-boundary-minimizers}
\end{equation}

\emph{Relaxation.} Let $L$ be the moments of $(0,0,T,T,S/2)$, where $S$ is a uniform random sign independent of $T$. This feasible functional attains the sum-of-squares lower bound zero, so $\widetilde\rho_2=\widetilde f^\star=0$. Its selected clique measure is
\[
    \mu_{\Delta_1}=\tfrac12\delta_{(0,0,-1/2)}
                       +\tfrac12\delta_{(0,0,1/2)},
\]
which represents all moments through degree four and satisfies the assigned box constraints. The selected block has rank two; the remaining blocks are unchanged from Example~\ref{ex:ring-rank-one}.

For every optimal $L$, $L(x_1^2)=L(x_2^2)=0$. The argument above annihilates all rows containing $x_1$ or $x_2$, so $M_2^{B_{\Delta_1}}(L)=A_2^{B_{\Delta_1}}(0,0)$. The zero square of $z^2-1/4-x_1-x_2$ places its coefficient vector in the kernel, making the row of $z^2$ one quarter of the constant row. Only the rows of $1$ and $z$ remain independent. Hence the maximum ranks over $\mathcal F_2$ are two for the selected clique and one for its boundary.

For the displayed rank-two block, these kernel relations and the moment identities give the minimal face
\[
    \operatorname{cone}\bigl\{
       A_2^{\Delta_1}(0,0,-1/2),\,
       A_2^{\Delta_1}(0,0,1/2)\bigr\}.
\]
Its two generators are linearly independent, establishing simplicial regularity.

\emph{Extendability.} Theorem~\ref{thm:simplicial-face} guarantees at least one globally extendable candidate. The boundary certificate of Theorem~\ref{thm:boundary-certificate} certifies both, as confirmed by $\pi_{\Delta_1}(\widetilde X^\star)=\{(0,0,-1/2),(0,0,1/2)\}$. At the common boundary value $(0,0)$, the interior objective is $(z^2-1/4)^2$, minimized at both extracted values. Rank one of the entire clique block is therefore unnecessary for this certificate.

We next isolate the role of value exactness. The selected clique and its boundary have rank one over the entire optimal face, but the extracted atom is not part of any global minimizer.

\begin{example}[Rank one without value exactness]
\label{ex:gap-rank-one}
Fix $0<\lambda\leq\frac14$ and consider
\begin{equation}
\begin{aligned}
    f_\lambda^\star=\min_{x,y,u,v}\quad
        &f_\lambda(x,y,u,v)=(x-y)^2+\lambda y^2-u^2-v^2\\
    \text{s.t.}\quad
        &x\geq0,\qquad y\geq0,\\
        &1-x^2\geq0,\quad 1-y^2\geq0,\quad
         1-u^2\geq0,\quad 1-v^2\geq0,\\
        &g_{\sigma,\tau}(y,u,v):=
          1+\frac12y-\frac12
          (\sigma u+\tau v+\sigma\tau uv)\geq0,
          \quad (\sigma,\tau)\in\{-1,1\}^2,
\end{aligned}
\label{eq:gap-pop}
\end{equation}
with the cliques $\Delta_1=\{x,y\}$ and $\Delta_2=\{y,u,v\}$ and the local objectives $f_1=(x-y)^2+\lambda y^2$ and $f_2=-u^2-v^2$. Assign $x\geq0$, $y\geq0$, $1-x^2\geq0$, and $1-y^2\geq0$ to $\Delta_1$, and the remaining constraints to $\Delta_2$, together with $3-y^2-u^2-v^2\geq0$. The selected clique is $\Delta_1$, with boundary $B_{\Delta_1}=\{y\}$ and interior $I_{\Delta_1}=\{x\}$. This two-clique cover satisfies RIP.
\end{example}

\emph{Global problem.} Put $a=|u|$ and $b=|v|$. Choosing the signs of $u$ and $v$ in $g_{\sigma,\tau}\geq0$ gives $a+b+ab\leq2+y$. On $[0,1]^2$,
\[
    1+a+b+ab-2(a^2+b^2)
       =2a(1-a)+2b(1-b)+(1-a)(1-b)\geq0.
\]
Thus $u^2+v^2\leq(3+y)/2$, and
\begin{equation}
    f_\lambda\geq\lambda y^2-\frac{3+y}{2}
      =\lambda-2+(1-y)\Bigl(\frac12-\lambda(1+y)\Bigr)\geq\lambda-2,
    \label{eq:gap-lower-bound}
\end{equation}
Equality requires $y=1$, $x=y$, and $u^2=v^2=1$. These points are feasible, so
\begin{equation}
    f_\lambda^\star=\lambda-2,\qquad
    X^\star=\{(1,1,\sigma,\tau):(\sigma,\tau)\in\{-1,1\}^2\}.
    \label{eq:gap-minimizers}
\end{equation}

\emph{Relaxation.} Consider the functional
\begin{equation}
    L(p)=\frac14\sum_{s,t\in\{-1,1\}}p(0,0,s,t)
    \label{eq:gap-functional}
\end{equation}
restricted to the sparse moment indices. Its clique representing measures are
\begin{equation}
    \mu_1=\delta_{(0,0)},\qquad
    \mu_2=\frac14\sum_{s,t\in\{-1,1\}}\delta_{(0,s,t)}.
    \label{eq:gap-local-measures}
\end{equation}
Writing $v_d^\Delta$ for the monomial vector on $\Delta$, its moment and localizing blocks satisfy
\begin{equation}
    M_2^{\Delta_\ell}(L)
      =\int v_2^{\Delta_\ell}\bigl(v_2^{\Delta_\ell}\bigr)^{\!\top}d\mu_\ell\succeq0,
    \qquad
    M_1^{\Delta_\ell}(pL)
      =\int p\,v_1^{\Delta_\ell}\bigl(v_1^{\Delta_\ell}\bigr)^{\!\top}d\mu_\ell.
    \label{eq:gap-matrices}
\end{equation}
The moment blocks are positive semidefinite, and the shared $y$-marginals both equal $\delta_0$. Every assigned inequality other than $g_{\sigma,\tau}$ is nonnegative on its corresponding clique support. For the remaining localizers, put $w_{st}=v_1^{\Delta_2}(0,s,t)=(1,0,s,t)^\top$. Then
\begin{equation}
    M_1^{\Delta_2}(g_{\sigma,\tau}L)
    =\frac32\diag{1,0,1,1}-\frac12w_{\sigma\tau}w_{\sigma\tau}^\top
    =
    \begin{pmatrix}
        1 & 0 & -\sigma/2 & -\tau/2\\
        0 & 0 & 0 & 0\\
        -\sigma/2 & 0 & 1 & -\sigma\tau/2\\
        -\tau/2 & 0 & -\sigma\tau/2 & 1
    \end{pmatrix}\succeq0.
    \label{eq:gap-localizer}
\end{equation}
The displayed matrix has eigenvalues $0,0,3/2,3/2$, proving feasibility. However, each point $(0,s,t)$ violates $g_{s,t}$, so $\mu_2$ is not locally feasible; $\mu_1$ is.

For every feasible $\widetilde L$, the moment and box-localizing constraints give
\[
    \widetilde L(f_\lambda)+2
      =\widetilde L\bigl((x-y)^2\bigr)+\lambda\widetilde L(y^2)
       +\widetilde L(1-u^2)+\widetilde L(1-v^2)\geq0.
\]
Since $L(f_\lambda)=-2$, we obtain $\rho_2=-2=f_\lambda^\star-\lambda$. The relaxation gap is positive and can be arbitrarily small.

At any optimal $L$, all four terms in this lower bound vanish. Because $\lambda>0$, $L(y^2)=L((x-y)^2)=0$, and positive semidefiniteness also gives $L(x^2)=0$. The box localizers propagate these zero relations as in Example~\ref{ex:ring-rank-one}, yielding $M_2^{\Delta_1}(L)=A_2^{\Delta_1}(0,0)$ for every $L\in\mathcal F_2$. Thus the selected clique and its boundary have maximum rank one, the selected block is simplicially regular, and its full-degree locally feasible representation is $\delta_{(0,0)}$.

\emph{Extendability.} All hypotheses of Theorem~\ref{thm:simplicial-face}, Corollary~\ref{cor:rank-one-extendability}, and Theorem~\ref{thm:boundary-certificate} except value exactness hold. Nevertheless, $\pi_{\Delta_1}(X^\star)=\{(1,1)\}$ excludes the extracted atom $(0,0)$. Fixing this atom removes every global minimizer, showing that value exactness cannot be omitted even when the relaxation gap is arbitrarily small.

Finally, we show that the boundary rank-one condition cannot be omitted from the guarantee that every extracted atom is globally extendable. We retain value exactness and simplicial regularity, but obtain both an extendable and a non-extendable atom.

\begin{example}[Value exactness without a rank-one boundary]
\label{ex:exact-rank-two}
Consider
\begin{equation}
\begin{aligned}
    f^\star=\min_{(x,y,z,u,v,w)\in[-1,1]^6}\quad
      &(x-y)^2+\Bigl(z-\frac{x+y}{2}\Bigr)^2-u^2-v^2+(vw)^2+(wx)^2\\
    \text{s.t.}\quad
      &x\geq0,\qquad y\geq0,\qquad y^2-y=0,\\
      &g_{\sigma,\tau}(y,u,v)\geq0,\qquad(\sigma,\tau)\in\{-1,1\}^2,
\end{aligned}
\label{eq:exact-ring-pop}
\end{equation}
with $g_{\sigma,\tau}$ as in~\eqref{eq:gap-pop}, the ring $\Delta_1=\{x,y,z\}$, $\Delta_2=\{y,u,v\}$, $\Delta_3=\{v,w\}$, $\Delta_4=\{w,x\}$, and the local objectives $f_1=(x-y)^2+(z-\frac{x+y}{2})^2$, $f_2=-u^2-v^2$, $f_3=(vw)^2$, and $f_4=(wx)^2$. Assign each box constraint to every clique containing its variable, assign $x\geq0$, $y\geq0$, and $y^2-y=0$ to $\Delta_1$, and assign the four $g_{\sigma,\tau}\geq0$ to $\Delta_2$. The selected clique has boundary $B_{\Delta_1}=\{x,y\}$ and interior $I_{\Delta_1}=\{z\}$. RIP fails as in Example~\ref{ex:ring-rank-one}.
\end{example}

\emph{Global problem.} The identity
\begin{equation}
    f+2=(x-y)^2+\Bigl(z-\frac{x+y}{2}\Bigr)^2+(vw)^2+(wx)^2+(1-u^2)+(1-v^2)
    \label{eq:exact-ring-certificate}
\end{equation}
gives $f\geq-2$ on $K$. Equality requires $u^2=v^2=1$, $x=y=z$, and $w=0$. Then $g_{u,v}(y,u,v)=(y-1)/2\geq0$ and $y\leq1$ force $y=1$. These points are feasible, so
\begin{equation}
    f^\star=-2,\qquad
    X^\star=\{(1,1,1,\sigma,\tau,0):(\sigma,\tau)\in\{-1,1\}^2\}.
    \label{eq:exact-ring-minimizers}
\end{equation}

\emph{Relaxation.} For $\theta\in[0,1]$, let $L_\theta$ be the moments of $(B,B,B,U,V,0)$, where $B$ is Bernoulli with $\Pr(B=1)=\theta$ and $U$, $V$ are uniform random signs, all independent. At $\theta=0$, the $\Delta_2$ measure is $\mu_2$ from~\eqref{eq:gap-local-measures}, whose localizers are positive semidefinite by~\eqref{eq:gap-localizer}; the other clique measures are locally feasible. At $\theta=1$, the functional integrates the uniform measure on $X^\star$. Both are feasible, so $L_\theta=(1-\theta)L_0+\theta L_1$ is feasible and attains $-2$. Together with~\eqref{eq:exact-ring-certificate}, this proves $\rho_2=f^\star=-2$.

For any optimal $L$, the coefficient vectors of $q=x-y$ and $q=z-y$ lie in $\ker M_2^{\Delta_1}(L)$ because the selected squares vanish. For each $s\in\{x,y,z\}$, the box localizer gives $0\leq L((sq)^2)\leq L(q^2)=0$, so the vectors of $sq$ also lie in the kernel. Consequently, the rows of $x$ and $z$ equal the row of $y$, and every quadratic row equals the row of $y^2$. The moment equalities for $y^2-y=0$ identify the latter with the row of $y$. Writing $\theta=L(y)=L(y^2)\in[0,1]$ yields
\begin{equation}
    M_d^{\Delta_1}(L)
      =(1-\theta)\,v_d^{\Delta_1}(0,0,0)\,v_d^{\Delta_1}(0,0,0)^\top
       +\theta\,v_d^{\Delta_1}(1,1,1)\,v_d^{\Delta_1}(1,1,1)^\top,
    \qquad d=1,2.
    \label{eq:exact-ring-face}
\end{equation}
Thus the selected blocks over the optimal face form exactly the segment joining the two point evaluations. For $0<\theta<1$, the same kernel relations and moment identities give the minimal face
\[
    \operatorname{cone}\bigl\{
       A_2^{\Delta_1}(0,0,0),\,
       A_2^{\Delta_1}(1,1,1)\bigr\};
\]
at the endpoints it is a single ray. Every optimal selected block is therefore simplicially regular and has the full-degree locally feasible representation $(1-\theta)\delta_{(0,0,0)}+\theta\delta_{(1,1,1)}$, with zero weights omitted. The maximum ranks over $\mathcal F_2$ are two for both the selected clique and its boundary.

\emph{Extendability.} At $\theta=1/4$, extraction gives
\begin{equation}
    \mu_{\Delta_1}=\tfrac34\delta_{(0,0,0)}+\tfrac14\delta_{(1,1,1)},
    \label{eq:exact-ring-extracted}
\end{equation}
Both atoms minimize the clique objective at their own boundary values, but $\pi_{\Delta_1}(X^\star)=\{(1,1,1)\}$. Hence the atom of larger weight, $(0,0,0)$, is not globally extendable, whereas $(1,1,1)$ extends to all four global minimizers. Theorem~\ref{thm:simplicial-face} still guarantees at least one extendable candidate because the selected block has maximum rank and satisfies simplicial regularity. The boundary certificate cannot guarantee every atom here: its rank-one hypothesis fails, and fixing the dominant atom excludes every global minimizer.

\section{Progressive extraction for a unicycle trajectory problem}
\label{sec:unicycle}

In this section, we apply the boundary certificate to progressive extraction, give an auxiliary SDP for testing rank one over the optimal face, and evaluate the method on a unicycle trajectory problem.

\subsection{Progressive extraction and fixing}
\label{subsec:progressive-extraction}

Let $\xi$ denote the variables of a POP of the form~\eqref{eq:pop}. At round $r$, the coordinates indexed by $J_r$ have been fixed to $c_r$, with $J_0=\varnothing$. Substituting $\xi_{J_r}=c_r$ gives a reduced POP, to which we apply the following steps.
\begin{enumerate}
    \item \emph{Solve.} Obtain an optimal moment solution $y^{(r)}$ of the reduced sparse relaxation at order $\kappa$.
    \item \emph{Extract.} For a selected clique $\Delta$, use simplicial regularity to obtain atomic candidates. 
    \item \emph{Test.} Compute the boundary $B_\Delta^{(r)}$ from the current clique cover and screen for
\begin{equation}
        \operatorname{rank}M_\kappa^{B_\Delta^{(r)}}(y^{(r)})=1,
        \label{eq:unicycle-extraction-test}
    \end{equation}
    Then verify rank one over the optimal face using the auxiliary test.
    \item \emph{Fix.} Choose an extracted atom $a$ and a nonempty set $A_r\subseteq\Delta$ of remaining coordinates. Fix $\xi_{A_r}=a_{A_r}$, update $J_{r+1}=J_r\cup A_r$ and the assigned values $c_{r+1}$, and repeat on the reduced problem.
\end{enumerate}
When the current relaxation is value exact, the optimal-face boundary condition allow any extracted atom to be fixed while preserving a global minimizer of the reduced POP (Theorem~\ref{thm:boundary-certificate}). If no clique is accepted, the numerical procedure uses a heuristic fallback.

\subsection{Testing rank one over the optimal face}
\label{subsec:optimal-face-test}

Rank one at the returned solution does not establish the condition~\eqref{eq:boundary-rank-one-condition} over the whole optimal face. We therefore test the largest possible deviation from its first-moment values. All quantities below refer to the current reduced problem, and we suppress the round index. Let $\mathcal S_\kappa$, $\rho_\kappa$, and $\mathcal F_\kappa$ denote its feasible moment set, optimal value, and optimal face, respectively, as in Section~\ref{subsec:rank-one-extendability}.

For a nonempty set $J$ contained in a clique and $1\leq t\leq\kappa$, choose $\bar y\in\mathcal F_\kappa$, set $c_j=L_{\bar y}(\xi_j)$ for $j\in J$, and define
\begin{equation}
    \Phi_{J,t}(\xi;c)=\sum_{\alpha\in\mathbb N_t^J\setminus\{0\}}(\xi^\alpha-c^\alpha)^2,
    \qquad c^\alpha=\prod_{j\in J}c_j^{\alpha_j}.
    \label{eq:deviation-polynomial}
\end{equation}
Positive semidefiniteness of $M_t^J(y)$ gives $L_y(\Phi_{J,t})\geq0$. We maximize this deviation over the optimal face:
\begin{equation}
\begin{aligned}
    \Theta_{\kappa,t}^J(c)=\sup_y\quad &L_y\bigl(\Phi_{J,t}(\cdot\,;c)\bigr)\\
    \text{s.t.}\quad &y\in\mathcal S_\kappa,\\
    &L_y(f)=\rho_\kappa.
\end{aligned}
\label{eq:rank-one-test-sdp}
\end{equation}
This retains the original relaxation constraints, changes the objective, and adds one scalar equality. 

\begin{proposition}[An optimal-face rank-one test]
\label{prop:rank-one-test}
Let $\bar y\in\mathcal F_\kappa$ and define $c$ and $\Theta_{\kappa,t}^J(c)$ as above. Then
\[
    \Theta_{\kappa,t}^J(c)=0
    \quad\Longleftrightarrow\quad
    \max_{y\in\mathcal F_\kappa}\operatorname{rank}M_t^J(y)=1.
\]
In this case, every optimal block is the same point-evaluation matrix:
\[
    M_t^J(y)=v_t^J(c)\,v_t^J(c)^\top\qquad(y\in\mathcal F_\kappa).
\]
The equivalence does not require value exactness.
\end{proposition}

\begin{proof}
If $\Theta_{\kappa,t}^J(c)=0$, every squared term in~\eqref{eq:deviation-polynomial} has zero pseudo-expectation for every $y\in\mathcal F_\kappa$. Positive semidefiniteness gives
\[
    M_t^J(y)(e_\alpha-c^\alpha e_0)=0,
    \qquad \alpha\in\mathbb N_t^J\setminus\{0\},
\]
where $e_\alpha$ is the coordinate vector for $\xi^\alpha$. Thus every column is a prescribed multiple of the constant column, and $y_0=1$ forces $M_t^J(y)=v_t^J(c)v_t^J(c)^\top$.

Conversely, if the maximum rank is one, Lemma~\ref{lem:common-range} gives a common one-dimensional range for all optimal blocks. Normalization makes the blocks identical, and the Hankel identities identify them with the point-evaluation matrix at $c$. Every squared deviation therefore has zero pseudo-expectation.
\end{proof}

For the boundary certificate, take $J=B_\Delta$ and $t=\kappa$. If the returned boundary block already has rank greater than one, the condition fails without an auxiliary solve. 
Otherwise, a numerical zero value in~\eqref{eq:rank-one-test-sdp} verifies the condition over the whole optimal face.

\subsection{Unicycle trajectory optimization}
\label{subsec:unicycle-results}

\paragraph{Problem and implementation.}
A planar unicycle travels toward a goal over $N$ steps of length $h=0.1$ while avoiding three circular obstacles. Representing the heading on the unit circle makes the discretized dynamics polynomial. The natural chain cover is $\Delta_k=\{s_{k-1},u_{k-1},s_k,q_k\}$, $k\in[N]$, where $s_k$ contains the position and heading, $u_k=(v_k,\omega_k)$ is the control, and $q_k$ is a heading increment. The formulation and experimental setup are given in Appendices~\ref{app:unicycle-formulation} and~\ref{app:unicycle-setup}.

Each round solves the order-two relaxation of the reduced problem. We fix controls directly and substitute any increments and states determined by these fixings. The boundary of $\Delta_k$ contains the states among $s_{k-1}$ and $s_k$ shared with other current cliques. 

We examine at most three screened cliques per round, in time-step order. For a nonempty boundary, we solve~\eqref{eq:rank-one-test-sdp} with $J=B_\Delta$ and $t=\kappa$, accepting the clique only if the computed value is at most $10^{-2}$.
If no candidate is accepted, the earliest unfixed control pair is fixed to its clique's dominant atom as a fallback, meaning the atom of largest weight. 

\paragraph{Evaluation.}
We compare progressive extraction with one-shot extraction, which solves the full order-two relaxation once and extracts one value per variable using SPOT. Each recovered trajectory $\widehat\xi$ is evaluated in the original POP. With the order-three lower bound $\rho_3\leq f^\star$, we report the relative suboptimality gap
\begin{equation}
    \eta=\frac{f(\widehat\xi)-\rho_3}{1+|f(\widehat\xi)|+|\rho_3|}.
    \label{eq:unicycle-gap}
\end{equation}
For a valid lower bound, a negative gap indicates infeasibility. We also report the largest equality residual, the largest inequality violation, and the number of steps penetrating an obstacle safety margin. Table~\ref{tab:unicycle-results} summarizes these metrics for $N\in\{40,60,80\}$, and Fig.~\ref{fig:unicycle} shows the trajectories.

\begin{table}[h]
\centering
\small
\caption{Comparison of one-shot and progressive extraction on the unicycle problem, both at relaxation order two on the chain cover. $\eta$: gap~\eqref{eq:unicycle-gap} to the order-three bound $\rho_3$. eq.\ and ineq.: largest equality residual and inequality violation. obst.: number of steps at which the safety margin of an obstacle is penetrated by more than $1$\,mm.}
\label{tab:unicycle-results}
\begin{tabular}{@{}rlrrrr@{}}
\toprule
$N$ & method & $\eta$ [\%] & eq. & ineq. & obst.\\
\midrule
40 & one-shot    & $0.00$  & $5.8\times10^{-5}$  & $3.2\times10^{-4}$ & 0 \\
   & progressive & $0.00$  & $6.7\times10^{-16}$ & $2.0\times10^{-5}$ & 0 \\
\midrule
60 & one-shot    & $-1.55$ & $3.0\times10^{-1}$  & $8.1\times10^{-1}$ & 34\\
   & progressive & $0.22$  & $6.7\times10^{-16}$ & $4.4\times10^{-5}$ & 0 \\
\midrule
80 & one-shot    & $-1.63$ & $2.8\times10^{-1}$  & $8.6\times10^{-1}$ & 34\\
   & progressive & $0.26$  & $2.2\times10^{-15}$ & $1.4\times10^{-4}$ & 0 \\
\bottomrule
\end{tabular}
\end{table}

\begin{figure}[h]
\centering
\includegraphics[width=\textwidth]{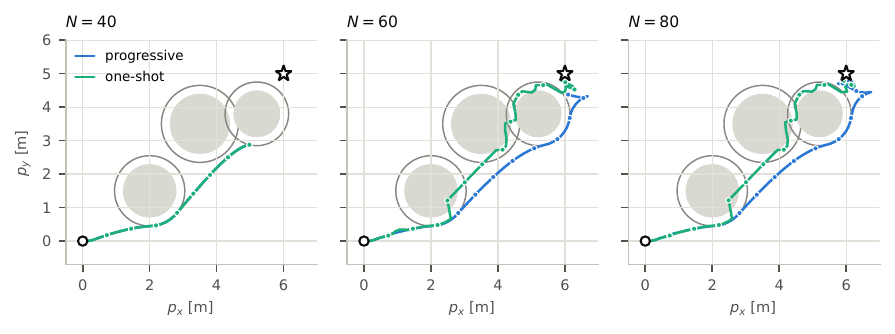}
\caption{Trajectories recovered by progressive and one-shot extraction for $N=40$, $60$, and $80$. Shaded disks are the obstacles, rings mark the safety margin, and dots are placed every five steps.}
\label{fig:unicycle}
\end{figure}

Progressive extraction successfully recovers almost optimal trajectories that satisfy the constraints. At longer horizons, one-shot extraction produces infeasible trajectories, while progressive fixing maintains the quality of the recovered solutions.

\section{Conclusion}
\label{sec:conclusion}

We established sufficient conditions for extending local atomic candidates to global minimizers in value-exact sparse moment relaxations. 
Under value exactness, simplicial regularity of a block with maximum rank over the optimal face guarantees at least one globally extendable candidate.
A full-degree locally feasible clique representation and a boundary block of rank one over that face guarantee extendability of the extracted atom.
An auxiliary SDP tests the boundary rank condition and guides progressive extraction. On a unicycle trajectory problem, progressive extraction improves feasibility over one-shot extraction while recovering near-optimal trajectories.

\section*{Acknowledgements}
I thank Heng Yang for helpful discussions on practical tests of the rank-one condition and on the examples, and Shucheng Kang for insightful discussions on simplicial singularity.

\bibliographystyle{plainnat}
\bibliography{references}

\begin{thebibliography}{23}
\providecommand{\natexlab}[1]{#1}
\providecommand{\url}[1]{\texttt{#1}}
\expandafter\ifx\csname urlstyle\endcsname\relax
  \providecommand{\doi}[1]{doi: #1}\else
  \providecommand{\doi}{doi: \begingroup \urlstyle{rm}\Url}\fi

\bibitem[Curto and Fialkow(2005)]{curto2005truncated}
Ra{\'u}l~E. Curto and Lawrence~A. Fialkow.
\newblock Truncated {$K$}-moment problems in several variables.
\newblock \emph{Journal of Operator Theory}, 54\penalty0 (1):\penalty0
  189--226, 2005.

\bibitem[Fantuzzi and Fuentes(2025)]{fantuzzi2025finite}
Giovanni Fantuzzi and Federico Fuentes.
\newblock Finite convergence and minimizer extraction in moment relaxations
  with correlative sparsity.
\newblock \emph{arXiv preprint arXiv:2502.01410}, 2025.

\bibitem[Henrion and Lasserre(2005)]{henrion2005}
Didier Henrion and Jean~B. Lasserre.
\newblock Detecting global optimality and extracting solutions in {GloptiPoly}.
\newblock In Didier Henrion and Andrea Garulli, editors, \emph{Positive
  Polynomials in Control}, volume 312 of \emph{Lecture Notes in Control and
  Information Sciences}, pages 293--310. Springer, 2005.

\bibitem[Kang and Yang(2026)]{kang2026simplicial}
Shucheng Kang and Heng Yang.
\newblock Simplicial regularizability of the pseudo-moment cone and
  {Carath\'eodory}-type atomic decomposition of moment matrices.
\newblock \emph{arXiv preprint arXiv:2605.06854}, 2026.

\bibitem[Kang et~al.(2024)Kang, Xu, Sarva, Liang, and Yang]{kang2024strom}
Shucheng Kang, Xiaoyang Xu, Jay Sarva, Ling Liang, and Heng Yang.
\newblock Fast and certifiable trajectory optimization.
\newblock \emph{arXiv preprint arXiv:2406.05846}, 2024.

\bibitem[Kang et~al.(2025)Kang, Liu, and Yang]{kang2025spot}
Shucheng Kang, Guorui Liu, and Heng Yang.
\newblock Global contact-rich planning with sparsity-rich semidefinite
  relaxations.
\newblock \emph{arXiv preprint arXiv:2502.02829}, 2025.

\bibitem[Klep et~al.(2018)Klep, Povh, and Vol{\v{c}}i{\v{c}}]{klep2018}
Igor Klep, Janez Povh, and Jurij Vol{\v{c}}i{\v{c}}.
\newblock Minimizer extraction in polynomial optimization is robust.
\newblock \emph{SIAM Journal on Optimization}, 28\penalty0 (4):\penalty0
  3177--3207, 2018.

\bibitem[Lasserre(2001)]{lasserre2001}
Jean~B. Lasserre.
\newblock Global optimization with polynomials and the problem of moments.
\newblock \emph{SIAM Journal on Optimization}, 11\penalty0 (3):\penalty0
  796--817, 2001.

\bibitem[Lasserre(2006)]{lasserre2006sparse}
Jean~B. Lasserre.
\newblock Convergent {SDP}-relaxations in polynomial optimization with
  sparsity.
\newblock \emph{SIAM Journal on Optimization}, 17\penalty0 (3):\penalty0
  822--843, 2006.

\bibitem[Lasserre et~al.(2008)Lasserre, Henrion, Prieur, and
  Tr{\'e}lat]{lasserre2007control}
Jean~B. Lasserre, Didier Henrion, Christophe Prieur, and Emmanuel Tr{\'e}lat.
\newblock Nonlinear optimal control via occupation measures and
  {LMI}-relaxations.
\newblock \emph{SIAM Journal on Control and Optimization}, 47\penalty0
  (4):\penalty0 1643--1666, 2008.

\bibitem[Magron and Wang(2021)]{magron2021tssos}
Victor Magron and Jie Wang.
\newblock {TSSOS}: a {Julia} library to exploit sparsity for large-scale
  polynomial optimization.
\newblock \emph{arXiv preprint arXiv:2103.00915}, 2021.

\bibitem[{MOSEK ApS}(2025)]{mosek}
{MOSEK ApS}.
\newblock \emph{{MOSEK} Optimizer {API} for {Python}, version 11.1}, 2025.
\newblock \url{https://www.mosek.com}.

\bibitem[Nie(2013)]{nie2013flat}
Jiawang Nie.
\newblock Certifying convergence of {Lasserre's} hierarchy via flat truncation.
\newblock \emph{Mathematical Programming}, 142\penalty0 (1--2):\penalty0
  485--510, 2013.

\bibitem[Nie et~al.(2026)Nie, Qu, Tang, and Zhang]{nie2026characterization}
Jiawang Nie, Zheng Qu, Xindong Tang, and Linghao Zhang.
\newblock A characterization for tightness of the sparse {Moment-SOS}
  hierarchy.
\newblock \emph{Mathematical Programming}, 215\penalty0 (1--2):\penalty0
  369--405, 2026.

\bibitem[Papachristodoulou and Prajna(2002)]{papachristodoulou2002lyapunov}
Antonis Papachristodoulou and Stephen Prajna.
\newblock On the construction of {Lyapunov} functions using the sum of squares
  decomposition.
\newblock In \emph{Proceedings of the 41st {IEEE} Conference on Decision and
  Control}, volume~3, pages 3482--3487, 2002.

\bibitem[Parrilo(2003)]{parrilo2003semidefinite}
Pablo~A. Parrilo.
\newblock Semidefinite programming relaxations for semialgebraic problems.
\newblock \emph{Mathematical Programming}, 96\penalty0 (2):\penalty0 293--320,
  2003.

\bibitem[Waki et~al.(2006)Waki, Kim, Kojima, and Muramatsu]{waki2005}
Hayato Waki, Sunyoung Kim, Masakazu Kojima, and Masakazu Muramatsu.
\newblock Sums of squares and semidefinite program relaxations for polynomial
  optimization problems with structured sparsity.
\newblock \emph{SIAM Journal on Optimization}, 17\penalty0 (1):\penalty0
  218--242, 2006.

\bibitem[Wang et~al.(2021{\natexlab{a}})Wang, Magron, and
  Lasserre]{wang2021chordal}
Jie Wang, Victor Magron, and Jean~B. Lasserre.
\newblock Chordal-{TSSOS}: A moment-{SOS} hierarchy that exploits term sparsity
  with chordal extension.
\newblock \emph{SIAM Journal on Optimization}, 31\penalty0 (1):\penalty0
  114--141, 2021{\natexlab{a}}.

\bibitem[Wang et~al.(2021{\natexlab{b}})Wang, Magron, and
  Lasserre]{wang2021tssos}
Jie Wang, Victor Magron, and Jean~B. Lasserre.
\newblock {TSSOS}: A moment-{SOS} hierarchy that exploits term sparsity.
\newblock \emph{SIAM Journal on Optimization}, 31\penalty0 (1):\penalty0
  30--58, 2021{\natexlab{b}}.

\bibitem[Wang et~al.(2022)Wang, Magron, Lasserre, and Mai]{wang2022cstssos}
Jie Wang, Victor Magron, Jean~B. Lasserre, and Ngoc Hoang~Anh Mai.
\newblock {CS-TSSOS}: Correlative and term sparsity for large-scale polynomial
  optimization.
\newblock \emph{ACM Transactions on Mathematical Software}, 48\penalty0
  (4):\penalty0 1--26, 2022.

\bibitem[Yang(2024{\natexlab{a}})]{yang2024moment}
Heng Yang.
\newblock {Semidefinite Optimization and Relaxation}, chapter 5: {Moment
  Relaxation}.
\newblock Online lecture notes,
  \url{https://hankyang.seas.harvard.edu/Semidefinite/Moment.html},
  2024{\natexlab{a}}.

\bibitem[Yang(2024{\natexlab{b}})]{yang2024structure}
Heng Yang.
\newblock {Semidefinite Optimization and Relaxation}, chapter 6: {Exploit
  Structure}.
\newblock Online lecture notes,
  \url{https://hankyang.seas.harvard.edu/Semidefinite/Structure.html},
  2024{\natexlab{b}}.

\bibitem[Yang and Carlone(2020)]{yang2020one}
Heng Yang and Luca Carlone.
\newblock One ring to rule them all: Certifiably robust geometric perception
  with outliers.
\newblock In \emph{Advances in Neural Information Processing Systems},
  volume~33, 2020.

\end{thebibliography}

\appendix
\section{Related work}
\label{app:related-work}

\paragraph{Polynomial optimization and the Moment--SOS hierarchy.}
Polynomial optimization provides a framework for nonconvex problems with polynomial objectives and constraints. 
Polynomial and sum-of-squares (SOS) methods have found applications in Lyapunov stability analysis~\citep{papachristodoulou2002lyapunov}, certifiable geometric perception~\citep{yang2020one}, and nonlinear optimal control through occupation measures~\citep{lasserre2007control}. 
Sparse moment relaxations have also enabled certifiable trajectory optimization~\citep{kang2024strom}. The Moment--SOS hierarchy, developed through the moment approach of \citet{lasserre2001} and the SOS formulations of \citet{parrilo2003semidefinite}, constructs a sequence of semidefinite programming relaxations of increasing order. 
The relaxation values form nondecreasing lower bounds on the global minimum and converge to it under an Archimedean assumption on the constraint description~\citep{lasserre2001}.

\paragraph{Solution extraction from moment relaxations.}
The flat-extension theory of \citet{curto2005truncated} connects rank stability of moment matrices with finitely atomic representing measures. 
Building on this theory, \citet{nie2013flat} studied flat truncation as a certificate for convergence and minimizer extraction. 
For an optimal moment sequence, the condition $\operatorname{rank} M_t(y)=\operatorname{rank} M_{t-d_c}(y)$ at a suitable truncation order $t$ certifies a finitely atomic measure supported on the feasible set.
Every atom is then a global minimizer. These atoms can be computed using multiplication matrices and eigenvalue decompositions, as implemented in GloptiPoly~\citep{henrion2005}. More recently, \citet{kang2026simplicial} proposed an atomic decomposition method based on simplicial regularizability of the pseudo-moment cone. We use the corresponding point-evaluation facial condition in Definition~\ref{def:simplicial-face} for local extraction, without requiring flatness. The resulting candidates must still be checked against the assigned clique constraints before their decomposition is called locally feasible.

\paragraph{Sparse Moment--SOS relaxations.}
Sparsity reduces the computational cost of the hierarchy. \emph{Correlative sparsity} exploits interactions among variables, replacing the dense moment matrix by smaller blocks associated with variable cliques~\citep{waki2005}. 
Under clique-wise Archimedean assumptions and the running intersection property (RIP), the sparse hierarchy retains asymptotic convergence to the global minimum~\citep{lasserre2006sparse}. 
\emph{Term sparsity} exploits the monomial supports of the input polynomials to construct smaller semidefinite blocks~\citep{wang2021tssos,wang2021chordal}. 
CS-TSSOS framework~\citep{wang2022cstssos} combines two forms of the sparsity to utilize the clique structure and further reduce the computational cost.

\paragraph{Solution extraction from sparse moment relaxations.}
For sparse moment relaxations, extracting global minimizers requires compatibility. 
\citet{lasserre2006sparse} established a sufficient condition combining flatness of every clique moment matrix with rank-one moment matrices on pairwise overlap. 
Subsequent results allow higher-rank overlaps. \citet{nie2026characterization} give sufficient conditions under RIP requiring flat clique truncations, a common rank across clique moment matrices, and flat overlap moment matrices with that same rank. 
\citet{fantuzzi2025finite} establish a sparse flat-extension theorem and a constructive extraction procedure under RIP and flatness conditions on every clique and the separators associated with an RIP ordering, without requiring equal clique ranks. 

\section{Unicycle trajectory problem}
\label{app:unicycle}

We give the polynomial formulation of the unicycle problem of §\ref{sec:unicycle} (Appendix~\ref{app:unicycle-formulation}) and the setup of the experiments (Appendix~\ref{app:unicycle-setup}) in this section.

\subsection{Polynomial formulation}
\label{app:unicycle-formulation}

A planar unicycle is steered over $N$ steps of length $h$ from a fixed initial state toward a goal while avoiding three circular obstacles (Fig.~\ref{fig:unicycle}). Following the lifting used for trajectory optimization with moment relaxations~\citep{kang2024strom,kang2025spot}, the heading is represented by a point on the unit circle rather than by an angle, which makes the discretized dynamics polynomial.

\paragraph{Variables.}
The state at step $k$ is $s_k=(p_k,r_k)$, where $p_k\in\mathbb R^2$ is the position and $r_k=(r_{c,k},r_{s,k})=(\cos\theta_k,\sin\theta_k)$ the heading, for $k=0,\ldots,N$. The control at step $k$ is $u_k=(v_k,\omega_k)$, a forward speed and a turn rate, for $k=0,\ldots,N-1$. The heading change over one step is carried by an auxiliary unit vector $q_k=(q_{c,k},q_{s,k})=(\cos\phi_k,\sin\phi_k)$, $k=1,\ldots,N$. The variable $\xi$ collects all states, controls, and increments; the initial state $s_0$ is fixed.

\paragraph{Dynamics.}
Discretizing the unicycle kinematics on the rotation group gives, for $k=0,\ldots,N-1$,
\begin{equation}
\begin{aligned}
    p_{x,k+1}&=p_{x,k}+h\,v_k\,r_{c,k},
    &\qquad
    r_{c,k+1}&=r_{c,k}\,q_{c,k+1}-r_{s,k}\,q_{s,k+1},\\
    p_{y,k+1}&=p_{y,k}+h\,v_k\,r_{s,k},
    &\qquad
    r_{s,k+1}&=r_{s,k}\,q_{c,k+1}+r_{c,k}\,q_{s,k+1},\\
    q_{s,k+1}&=h\,\omega_k,
\end{aligned}
\label{eq:unicycle-polynomial-dynamics}
\end{equation}
together with the unit-circle constraints $r_{c,k}^2+r_{s,k}^2=1$ and $q_{c,k}^2+q_{s,k}^2=1$. The inequality $q_{c,k}\geq q_{\min}$, with $q_{\min}>0$, selects the positive cosine branch and bounds the heading change per step.

\paragraph{Constraints and objective.}
The speed and the turn rate are bounded by $v_k^2\leq v_{\max}^2$ and $\omega_k^2\leq\omega_{\max}^2$, the positions lie in a box, and the obstacles with centers $c_j$ and radii $\varrho_j$ are avoided with a safety margin $\delta$,
\begin{equation}
    \|p_k-c_j\|^2\geq(\varrho_j+\delta)^2,
    \qquad j=1,2,3,\quad k=0,\ldots,N.
    \label{eq:unicycle-obstacles}
\end{equation}
The goal position $p_{\mathrm g}$ and heading $r_{\mathrm g}$ enter through a terminal and a running cost,
\begin{equation}
    f(\xi)=
    4\,\|p_N-p_{\mathrm{g}}\|^2+2\,\|r_N-r_{\mathrm{g}}\|^2
    +\sum_{k=1}^{N}\bigl(\|p_k-p_{\mathrm{g}}\|^2+\|r_k-r_{\mathrm{g}}\|^2\bigr),
    \label{eq:unicycle-objective}
\end{equation}
and there is no control penalty.

\paragraph{Sparsity.}
Every constraint couples the variables of at most two consecutive steps, and every term of $f$ involves a single step. The chain cover of Example~\ref{ex:chain} therefore becomes
\begin{equation}
    \Delta_k=\{s_{k-1},u_{k-1},s_k,q_k\},\qquad k\in[N],
    \label{eq:unicycle-chain-cover}
\end{equation}
and consecutive cliques overlap in $s_k$, so RIP holds. In the terminology of Definition~\ref{def:boundary}, $u_{k-1}$ and $q_k$ are interior variables of $\Delta_k$, and, for $1<k<N$, the states $s_{k-1}$ and $s_k$ form its boundary. All constraints have degree at most two, and the experiments use relaxation order $\kappa=2$.

\subsection{Experimental setup}
\label{app:unicycle-setup}

\paragraph{Relaxations.}
All relaxations are solved with MOSEK~\citep{mosek} through SPOT~\citep{kang2025spot}, with positions and controls scaled by their bounds to $[-1,1]$. 
The one-shot relaxation and every round of the progressive procedure are of order two on the chain cover~\eqref{eq:unicycle-chain-cover}.
Before each round, the fixed controls and the states they determine are substituted as constants, and the cliques are restricted to the remaining variables. The bound $\rho_3$ is the value of the order-three relaxation of the original problem.

\paragraph{Selection and propagation.}
On the chain cover, the control $u_{k-1}$ lies in exactly one clique, $\Delta_k$. Each round fixes the control of the earliest unfixed step whose clique passes the screening test~\eqref{eq:unicycle-extraction-test} and the numerical optimal-face test, or of the earliest unfixed step if no candidate is accepted, to the coordinates of the dominant atom of that clique. The increment and the next state then follow from $q_{s,k}=h\,\omega_{k-1}$, $q_{c,k}=\sqrt{1-h^2\omega_{k-1}^2}$ and the dynamics~\eqref{eq:unicycle-polynomial-dynamics}. The recovered trajectory is the forward propagation of the fixed controls.

\end{document}